\documentclass[a4paper,10pt]{article}
\usepackage[utf8]{inputenc}
\usepackage{amsthm}
\usepackage{amsmath}
\usepackage{relsize}
\usepackage{authblk}
\usepackage{mathtools}
\usepackage[mathscr]{euscript}
   \usepackage[pdftex]{graphicx}
   \usepackage[justification=centering]{caption}
   \usepackage{float}
\usepackage{nicefrac}
\usepackage{cite}
\usepackage{amssymb}
\usepackage{hyperref}
\newtheorem{theorem}{Theorem}
\newtheorem{claim}{Claim}
\newtheorem{lemma}{Lemma}

\newtheorem{proposition}{Proposition}
\newcommand{\pro}{\text{{\LARGE $\sqcap$}}}
\newtheorem{remark}{Remark}
\title{Asynchronous stochastic approximations with asymptotically biased errors and deep multi-agent learning
}
\author{Arunselvan~Ramaswamy\thanks{\textit{supported by the German Research Foundation (DFG) - 315248657.}}~\texttt{arunr@mail.uni-paderborn.de} \thanks{Department of Electrical Engineering and Information Technology, Universit\"{a}t Paderborn, Paderborn - 33098, Germany.} \and
      Shalabh~Bhatnagar~\texttt{shalabh@iisc.ac.in}\thanks{Department of Computer Science and Automation, Indian Institute of Science, Bangalore - 560012, India.} \and
      Daniel E. Quevedo~\texttt{dquevedo@ieee.org}\thanks{Department of Electrical Engineering and Information Technology, Universit\"{a}t Paderborn, Paderborn - 33098, Germany.} 
      }

\begin{document}
\maketitle
\begin{abstract}
Asynchronous stochastic approximations (SAs) are an important class of model-free algorithms, tools and techniques that are popular in multi-agent
and distributed control scenarios. To counter Bellman's curse of dimensionality, such algorithms are coupled with function approximations. Although the learning/ control problem becomes more tractable, function approximations affect stability and convergence.
In this paper, we present verifiable sufficient conditions for stability and convergence of asynchronous SAs with biased approximation errors. The theory developed herein is used to analyze
Policy Gradient methods
and noisy Value Iteration schemes. 
Specifically, we analyze the asynchronous approximate counterparts of the policy gradient (A2PG)
and value iteration (A2VI) schemes.
It is shown that the stability of these algorithms is unaffected by
biased approximation errors, provided they are
asymptotically bounded. With respect to convergence (of A2VI and A2PG), a relationship between the limiting set and the approximation errors
is established. Finally, experimental results are presented that support the theory.
\end{abstract}

\section{Introduction}\label{sec_introduction}
In recent years reinforcement learning (RL) algorithms such as Value Iteration, Q-learning
and Policy Gradient methods have witnessed a colossal resurgence. Many of these algorithms are coupled with function approximators to solve many important
problems including, but not limited to, autonomous driving in transportation, 
process optimization in industrial scenarios and
efficient dispersal of health-care services. A neural network with several hidden layers is called a deep neural network (DNN).
RL that uses a DNN for function approximation is called DeepRL. The literature around DeepRL is growing rapidly, for example see
\cite{mnih}, \cite{mniha},  \cite{tamar} and \cite{li17}.

Most modern learning and control problems have continuous state and/ or action spaces. This leads to {\bf Bellman's curse of dimensionality}. To overcome this curse, learning and control algorithms are coupled with function approximation. It is worth noting that the previously mentioned resurgence is partly owing to the effectiveness of DNN in function approximation.  While the problem becomes tractable, the use of function approximation affects stability (almost sure boundedness) and convergence properties of the algorithms. Further, the optimality of the policies found depends on the {\bf approximation errors}. Such issues are not well studied. {\bf The main contribution of this paper is a complete analysis in terms of the influence of function approximation on stability and the limiting set, in a multi-agent setting}.
%An important drawback of using function approximation is that one can only expect to find
%suboptimal solutions. In such cases, it is imperative to completely characterize the behavior of
%DeepRL algorithms, understand the effect of function approximation and provide guarantees on
%convergence and stability (almost sure boundedness of the algorithm).

While the theory behind traditional RL is mature, there have not been many attempts to analyze DeepRL. Munos analyzed the approximate value
and policy iteration algorithms, see \cite{munos} and \cite{munos03}. However the assumptions in \cite{munos} and \cite{munos03} are rather restrictive. 
Ramaswamy and Bhatnagar \cite{ramaswamy2017} studied approximate value iteration methods under significantly weaker assumptions as compared to \cite{munos}. However, \cite{ramaswamy2017} does not consider the multi-agent scenario.
% In
% \cite{ramaswamy2017} sufficient conditions, tailor made for DeepRL applications, were presented that guarantee
% convergence and almost sure boundedness of the iterates. Specifically, they showed that approximate
% value iteration converges to a fixed point of the perturbed Bellman operator. Further,
% they showed that asymptotically 
% bounded-error function approximators have a stabilizing effect on the algorithm. 
In this paper, we present the framework to develop and analyze large-scale multi-agent RL algorithms.
Such algorithms are applicable to industrial process-control, distributed control
of microgrids and decentralized resource allocation systems, among others. \textit{It may be noted that in the setting of distributed control 
and learning, the aforementioned curse of dimensionality problem is particularly pronounced}.

\subsection{Motivation, relevant literature and our contributions}\label{sec_contributions}
The main motivation for this paper is the development of a general framework for learning and control in large-scale multi-agent settings. In a typical multi-agent architecture, the agents involved need to work towards a common goal by cooperating with each other. Each agent may be asynchronously performing updates, i.e., according its own local clock, but not that of other agents. While the agents may act independently, their decisions are based on information from other agents. This information is often shared via wireless communication networks. Bottlenecks in communication resources may lead to (possibly unbounded) delays as well as errors in communication. 
% It is worth noting that in the current work we do not assume the existence of a global clock for implementation.
% 
% In this paper, we provide a complete analysis of the asynchronous approximate value iteration algorithm
% (A2VI). This is a stochastic iterative multi-agent learning algorithm that
% is an adaptation of traditional value iteration to the setting
% of optimal control of large-scale distributed systems and also to large-scale multi-agent learning.
% For the analysis we use the lens of stochastic approximation algorithms and utilize tools developed
% in the field of viability theory. Recall that
% the agents are all assumed to be running their own local clocks and communicate with each other
% using communication channels that are prone to delays and errors. An analysis of 
% asynchronous Q-learning, for this setting, was done by Abounadi, Bertsekas and Borkar in 2002, see \cite{abounadi}.
% However, the analysis in \cite{abounadi} does not 
% account for the use of function approximators, a necessity for many large-scale applications. 
% In our recent work \cite{ramaswamy2017}, an analysis is provided
% for value iterations with function approximations; however \cite{ramaswamy2017}
% does not account for multi-agent distributed scenarios. In the current paper
% we provide a complete analysis (stability and convergence) of the previously described A2VI. 

Here, we focus on developing a framework which accounts for all of the above constraints. Such a framework would then guide the
development of algorithms with behavioural guarantees. For this, we build on tools and techniques developed in \cite{Borkar_asynchronous}, \cite{abounadi}, \cite{perkins}, \cite{Benaim96}, \cite{Benaim05}, \cite{ramaswamy2017} and \cite{aubin2012differential}. Traditional analyses in \cite{Borkar_asynchronous} and \cite{abounadi} do not account for the use of function approximation. Whist the more modern analysis in \cite{perkins} can be modified to the multi-agent setting considered herein. However, \cite{perkins} does not consider the important question of {\bf algorithmic stability}. Further, the results of \cite{perkins} do not characterize the limiting set as a function of the approximation errors.

DNN is a popular choice for function approximation among practitioners of RL. They are used to approximate objective functions such as Q-factors, value functions, policies, gradients, etc. Such a DNN is trained in an online manner to minimize its prediction errors.  The network architecture is typically chosen without explicit knowledge of the objective function. Hence, one cannot expect the approximate objective function obtained using a DNN to equal the true objective function everywhere, even after sufficient training. In other words, the approximation errors are likely to be {\bf biased} (have non-zero mean). Further, these biases may affect the stability and quality of convergence, see Remark~\ref{asmp_remark1} for details.

The main contributions of this paper can be summarized as follows:
\begin{enumerate}
\item We show that the stability of the algorithms remains unaffected by error biases, provided error biases do not grow over time.
\item We provide an explicit relation between the biases and the limiting set.
\item Although we consider approximation error biases, there are no additional restrictions on the quality of communication. In other words, we present our results under standard (yet general) assumptions on communication as used, e.g., in \cite{Borkar_asynchronous}.
\item Our theory is used to analyze the asynchronous approximate counterpart of value (A2VI) and policy gradient (A2PG) iterations.
\end{enumerate}

\subsection{Tool set: Asynchronous Stochastic Approximations}
Stochastic approximation algorithms (SAs) encompass a class of iterative algorithms that are model-free and sample-based. SAs find the minimum/maximum of a given objective function through a series of approximations. Traditionally, the approximation errors are expected to vanish over time. In 1951 the first SA was developed by Robbins and Monro \cite{robbins} 
for finding a root of a given regression function. The theory of modern SAs was developed by Bena\"{i}m \cite{Benaim96}, Bena\"{i}m
and Hirsch \cite{BenaimHirsch} and Borkar \cite{Borkar99}. This theory was extended to SAs with set-valued mean-fields
by Bena\"{i}m, Hofbauer and Sorin \cite{Benaim05} \cite{benaim2006stochastic}, Ramaswamy and Bhatnagar \cite{Ramaswamy}, Perkins and Leslie \cite{perkins}, Bianchi et. al. \cite{bianchi2019constant}, and others. The reader is referred to books by
Borkar \cite{BorkarBook} and Kushner and Yin \cite{KushnerYin}
for a more detailed exposition on this topic.

Although the traditional SA framework can be used to develop and analyze algorithms in RL and stochastic control, it does not encompass
multi-agent and distributed scenarios. The latter setting was studied by Borkar \cite{Borkar_asynchronous}. He considered multi-agent algorithms wherein the agents are asynchronous and communications are delayed/ erroneous. Such algorithms are called asynchronous SAs. Many RL
algorithms such as Q-learning, value iteration and policy gradient methods have asynchronous counterparts. These algorithms are designed and analyzed using the theory developed in \cite{Borkar_asynchronous} and \cite{abounadi}. The stability issue of asynchronous SAs was studied by Bhatnagar \cite{Bhatnagar}. 

%In deep multi-agent learning applications, in addition to distributed
%scenarios, one often uses function approximations to allow for large state and action spaces. 
%It is unreasonable to expect these errors to vanish, see Remark~\ref{asmp_remark1} at the end of Section~\ref{sec_assumptions} for details. To this end, we extend the framework of Borkar 
%\cite{Borkar_asynchronous} and Bhatnagar \cite{Bhatnagar} to account for approximation errors that are %possibly non-diminishing and biased. This results in an asynchronous stochastic
%approximation algorithm with asymptotically bounded, and possibly biased, errors. The reader is referred to equation (\ref{asmp_aasaa}) in Section~\ref{sec_assumptions}
%for the recursion. Note that we present assumptions for both stability and convergence of (\ref{asmp_aasaa}). We use the aforementioned extension to analyze the asynchronous
%approximate counterpart of value (A2VI) and policy gradient (A2PG) iterations. These are two simple yet effective reinforcement learning algorithms, which we briefly
%discuss in the next subsection.
\subsection{Value Iteration and Policy Gradient for multi-agent settings} \label{intro_a2}
We are interested in the following adaptation of value iteration for the multi-agent setting:
%It may be noted that many of our notations are from Abounadi, Bertsekas and Borkar \cite{abounadi}. Below we state this adaptation.
\begin{equation}
 \label{intro_a2vi}
 \begin{split}
  &J_{n+1}(i) = J_n(i) + a( \nu (n, i)) I(i \in Y_n) \\ &\left[ (\mathcal{A} T)_i (J_{n - \tau_{1 i}(n)}(1), \ldots, J_{n - \tau_{d i}(n)}(d)) + 
 M_{n+1}(i) \right], n \ge 0.
 \end{split}
\end{equation}
In the above equation,\\
(i) $1 \le i \le d$ is the agent index, and there are $d$ agents in the system. \\
(ii) $J_n := (J_n(1), \ldots, J_n(d))$ is an estimate of the optimal cost-to-go vector at time-step $n$.
\\
(iii) $Y_n$ is a subset of $\{1, 2, \ldots, d\}$ for each $n \ge 0$. It represents the number of agents that are
\textbf{active} at time $n$.\\
(iv) $0 \le \tau_{ji}(n) \le n$ is the (stochastic) delay experienced by agent $i$ in receiving information from agent $j$
at time $n$. In other words, at time $n$, the information obtained by agent $i$ from agent $j$
is $\tau_{ji}(n)$ time-steps old.\\
(v) $\nu(n,i)$ is the number of times that agent $i$ was active (i.e., updated its component parameter) up until time $n$. {\it It may be noted that the time index $n$ in equation (\ref{intro_a2vi}) represents the global clock, and thus, grows unbounded. We analyze the algorithm with respect to this clock.}  Let us say that agent-2 has been updated $34$ times when the global clock has been updated $50$ times. Then we have $\nu(50,2) = 34$.
\\
(vi) $\mathcal{A}$ is the approximation operator (deep neural network), $\{a(n)\}_{n \ge 0}$ is the given step-size sequence and
$\{M_{n+1}\}_{n \ge 0}$ is a Martingale difference noise sequence.

Let us call recursion (\ref{intro_a2vi}) \textit{asynchronous approximate value iteration} (A2VI). If the optimal cost-to-go vector associated with agent-$i$ is $J^*(i)$, then
$J^* = (J^*(1), \dots, J^*(d))$ is the optimal cost-to-go vector associated with the $d$-agent system. The objective is to find $J^*$
in an ``asynchronous'' manner. 
%Although we assume that each agent runs its own local clock,
%we require that the agents perform updates, comparably often, see assumption $(S2)$ in Section~\ref{sec_stability_assumptions}for details. Recall that the agents exchange information with each other in order to achieve a common goal. 

At any step $n$, the information at agent-$i$ from agent-$j$ is $\tau_{ji}(n)$ steps old. The stochastic delay process $\tau$ could be unbounded. However, we make certain standard assumptions on their moments, see $(A2)(v)$ in Section~\ref{sec_delay}. Further,  we assume that the agent-updates are all in the same order of magnitude, asymptotically, see $(S2)$ in Section~\ref{sec_stability_assumptions}. Under these assumptions, we shall show that (\ref{intro_a2vi}) is stable and converges to a fixed point of the perturbed Bellman operator. The reader is referred to Section~\ref{sec_a2vi} for details. 

\textit{It is important to note that we do not distinguish
between stochastic shortest path (with no discounting) and infinite horizon discounted cost problems. Only the definition of the Bellman operator changes accordingly.} The reader is referred to \cite{BertsekasBook} for details on how the Bellman operator changes.

The {\bf policy gradient} algorithm is another important reinforcement learning approach, see \cite{sutton2000}. This method assumes a parameterization
$\theta$ of the policy space $\pi$. Finding an optimal policy amounts to finding a $\hat{\theta}$ that locally minimizes the parameterized policy function $\pi(\cdotp)$.
We are interested in adapting the policy gradient algorithm to the aforementioned multi-agent setting:
\begin{equation}
\label{intro_a2pi}
\begin{split}
 &\theta_{n+1}(i) = \theta_n (i)  - a(\nu(i,n)) I\{ i \in Y_n\} \\ &\left( (\mathcal{A}\nabla_\theta \pi)_i(\theta_{n - \tau_{1i}(n)}(1), \ldots, 
\theta_{n - \tau_{di}(n)}(d))
   + M_{n+1}(i) \right), n \ge 0.
\end{split}
\end{equation}
In the above equation,
$\theta$ is the parameter associated with policy $\pi$ and $\mathcal{A}$ is the approximation of the policy function gradient. There may be a multitude of reasons for using $\mathcal{A}$.
Most important among these is the non-availability of gradients, $\nabla_\theta \pi(\cdotp)$, due to the use of gradient estimators
or the non-differentiability of $\pi$. In the latter case, one may work with the sub-gradient and its approximations instead of gradient itself. Note that a slight visual
inspection reveals the similarity in the forms of recursions (\ref{intro_a2vi}) and (\ref{intro_a2pi}). We call (\ref{intro_a2pi})
\textit{asynchronous approximate policy gradient} (A2PG) algorithm, see Section~\ref{sec_a2pi} for details.
\subsection{Organization}
The organization of the remainder of this paper is as follows:
\begin{itemize}
\item In the following section, we list the definitions and notations used throughout. 
\item In Section~\ref{sec_assumptions} we present the assumptions involved in the analysis of
asynchronous stochastic approximations with asymptotically bounded biased, errors, \textit{i.e.,} recursion (\ref{asmp_aasaa}). 
\item In Sections~\ref{sec_nodelay}, \ref{sec_delay} and \ref{sec_balance}, we present a convergence
analysis of (\ref{asmp_aasaa}) under the assumptions presented in Section~\ref{sec_assumptions}.
The main technical result of this paper, Theorem~\ref{delay_main}, is enunciated in Section~\ref{sec_delay}. This result
is then moulded through the use of Borkar's balanced
step-sizes \cite{Borkar_asynchronous}, into the desired statement in Section~\ref{sec_balance}. 
\item In Section~\ref{sec_stability}, we show that
\textit{the stability of algorithms remains unaffected when the approximation errors are guaranteed to be asymptotically bounded albeit non-diminishing and possibly biased}.
\item In Section~\ref{sec_a2vi} our theory is used to understand the long-term behavior of A2VI. \textit{We show that
A2VI converges to a fixed point of the perturbed Bellman operator, when Borkar's balanced step-sizes are utilized.
We also establish a relationship between these fixed points and the approximation errors.}
\item In Section~\ref{sec_a2pi} we briefly outline a similar analysis for A2PG. \textit{We show that A2PG converges to a small neighborhood of local minima
of the parameterized policy function $\pi(\cdotp)$. This neighborhood is shown to be related to the approximation errors.} 
\item In Section~\ref{sec_exp} we present experimental results to support our theory.
\item In Section~\ref{sec_verify} we discuss the verifiability of assumption $(S5)$. Finally, we summarize our contributions
in Section~\ref{sec_conclusion}.
\end{itemize}
%%%%%%%%%%%%%%%%%%%%%%%%%%%%%%%%%%%%%%%%%%%%%%%%%%%%%%%%%%%%%%%%%%%%%%%%%%%%%%%%%%%%%%%%%%%%%%%%%%%%%%%%%%%%%%%%%%%%%%%%%%%%%%%%%%%%%%%%%%%%%%%%
%%%%%%%%%%%%%%%%%%%%%%%%%%%%%%%%%%%%%%%%%%%%%%%%%%%%%%%%%%%%%%%%%%%%%%%%%%%%%%%%%%%%%%%%%%%%%%%%%%%%%%%%%%%%%%%%%%%%%%%%%%%%%%%%%%%%%%%%%%%%%%%%
%%%%%%%%%%%%%%%%%%%%%%%%%%%%%%%%%%%%%%%%%%%%%%%%%%%%%%%%%%%%%%%%%%%%%%%%%%%%%%%%%%%%%%%%%%%%%%%%%%%%%%%%%%%%%%%%%%%%%%%%%%%%%%%%%%%%%%%%%%%%%%%%
\section{Definitions and Notations} \label{sec_definitions}
{\renewcommand\labelitemi{}
\subsection{General}
\begin{itemize}
\item \textbf{\textbf{[Set closure]}} Given $A \subset \mathbb{R}^d$, then $\overline{A}$ is used to represent the closure of $\mathcal{A}$.
\item \textbf{\textbf{[Limiting set]}} Given $\{ x_n \}_{n \ge 0} \subset \mathbb{R}^d$, its limiting set is given by $\underset{n \ge 0}{\bigcap} \overline{\{x_n \mid n \ge N\}}$.
\item \textbf{\textbf{[Distance between point and set]}} Given $x \in \mathbb{R}^d$ and $A \subseteq \mathbb{R}^d$, 
the distance between $x$ and $A$ is given by:
$d(x, A) : = \inf \{\lVert a- y \rVert \ | \ y \in A\}$.
\item \textbf{\textbf{[$\delta$-open neighborhood ($N^\delta(\cdotp)$)]}}
We define the $\delta$-\textit{open neighborhood}
of $A \subset \mathbb{R}^d$ by $N^\delta (A) := \{x \ |\ d(x,A) < \delta \}$. 
\item \textbf{\textbf{[Balls of radius $r$ $(B_r(0)$ and $\overline{B}_r(0))$]}}
The open ball of radius $r$ around the origin is represented by $B_r(0)$,
while the closed ball is represented by $\overline{B}_r(0)$.
\item \textbf{\textbf{[Projection map]}}
Given $\mathcal{B}$ and $\mathcal{C}$ subsets of $\mathbb{R}^d$, the projection map
$\pro _{\mathcal{B,C}}: \mathbb{R}^d \to \{\text{subsets of }\mathbb{R}^d\}$ is given by:
\[
	\pro _{\mathcal{B,C}}(x) :=  \begin{cases}
	\{x\} \text{, if $x \in \mathcal{C}$}  \\
	\{y \mid d(y, x) = d(x, \overline{\mathcal{B}}), \ y \in \overline{\mathcal{B}} \}  \text{, otherwise}.
	\end{cases}.
\]
\end{itemize}
\subsection{Related to norms and function spaces}
\begin{itemize}
\item \textbf{\textbf{[Euclidean norm ($\lVert \cdotp \rVert$)]}} Given $x \in \mathbb{R}^d$, $\lVert x \rVert$ is used to represent the Euclidean norm of $x$, i.e., $\lVert x \rVert = \sqrt{x_1 ^2 + \ldots + x_d ^2}$.
\item \textbf{[Weighted max-norm ($\lVert \cdotp \rVert_{\nu}$)]} Given $\nu = (\nu_1, \ldots, \nu_d)$ such that
$\nu_i > 0$ for $1 \le i \le d$,
the weighted max-norm of any $x = (x_1, x_2, \ldots, x_d) \in \mathbb{R}^d$ is given by:
$\lVert x \rVert_\nu := \max \left\{ \frac{|x_i|}{\nu_i}  \mid
1 \le i \le d \right\}$. \\
\item \textbf{[Weighted p-norm ($\lVert \cdotp \rVert_{\omega, p}$)]} Given $\omega = (\omega_1, \ldots, \omega_d)$ such that 
$\omega_i > 0$ for $1 \le i \le d$, and $p \ge 1$, the weighted p-norm of any $x \in \mathbb{R}^d$
is defined by:
 $\lVert x \rVert_{\omega, p} := \left( \sum \limits_{i=1}^d \lvert \omega_i x_i \rvert^p \right)^{1/p}.$ 
\item \textbf{\textbf{[Square integrable functions]}} $\mathbb{L}^2([0,T], \mathbb{R}^d)$ is used to represent the set of all square integrable functions with domain $[0,T]$ and range $\mathbb{R}^d$. In other words,
\[
\mathbb{L}^2([0,T], \mathbb{R}^d) = \left\{ f:[0,T] \to \mathbb{R}^d \ \mathlarger{\mathlarger{\mid}} \int_0 ^T \lVert f(t) \rVert ^2 dt < \infty \right\}.
\] 
\item \textbf{\textbf{[C\`{a}dl\`{a}g functions]}} $D([0,T], \mathbb{R}^d)$ is used to represent the set of all C\`{a}dl\`{a}g functions with domain [0,T] and range $\mathbb{R}^d$. This is the set of all functions that are right continuous with left limits.
\item \textbf{\textbf{[Lipschitz continuity]}} A function $f: \mathbb{R}^n \to \mathbb{R}^m$ is Lipschitz continuous {\it iff} $\exists L > 0$ such that $\forall x, y \in \mathbb{R}^n$ $\lVert f(x) - f(y) \rVert \le L \lVert x - y \rVert$.
 \item \textbf{\textbf{[Upper-semicontinuous map]}} We say that a set-valued map $H$ is upper-semicontinuous,
  if for given sequences $\{ x_{n} \}_{n \ge 1}$ (in $\mathbb{R}^{n}$) and 
  $\{ y_{n} \}_{n \ge 1}$ (in $\mathbb{R}^{m}$)  such that
  $x_{n} \to x$, $y_{n} \to y$ and $y_{n} \in H(x_{n})$, $n \ge 1$, 
  we have $y \in H(x)$.
  \end{itemize}
  \subsection{Related to differential inclusions}
  \begin{itemize}
\item
\textbf{\textbf{[Marchaud Map]}} A set-valued map $H: \mathbb{R}^n \to \{\text{subsets of }\mathbb{R}^m$\} 
is called \textit{Marchaud} if it satisfies
the following properties:
 \textbf{(i)} for each $x$ $\in \mathbb{R}^{n}$, $H(x)$ is a convex and compact set;
 \textbf{(ii)} \textit{(point-wise boundedness)} for each $x \in \mathbb{R}^{n}$,  
 $\underset{w \in H(x)}{\sup}$ $\lVert w \rVert$
 $< K \left( 1 + \lVert x \rVert \right)$ for some $K > 0$;
 \textbf{(iii)} $H$ is \textit{upper-semicontinuous}.
 
 Let $H$ be a Marchaud map on $\mathbb{R}^d$.
The differential inclusion (DI) given by
\begin{equation} \label{di}
\dot{x} \ \in \ H(x)
\end{equation}
is guaranteed to have at least one solution that is absolutely continuous. 
The reader is referred to \cite{aubin2012differential} for more details.
We say that $\textbf{x} \in \sum$ if $\textbf{x}$ 
is an absolutely continuous map that satisfies (\ref{di}).
The \textit{set-valued semiflow}
$\Phi$ associated with (\ref{di}) is defined on $[0, + \infty) \times \mathbb{R}^d$ as: \\
$\Phi_t(x) = \{\textbf{x}(t) \ | \ \textbf{x} \in \sum , \textbf{x}(0) = x \}$.\\ For
$B \times M \subset [0, + \infty) \times \mathbb{R}^d$, we define
\begin{equation}\nonumber
 \Phi_B(M) = \underset{t\in B,\ x \in M}{\bigcup} \Phi_t (x).
\end{equation}
\item \textbf{\textbf{[Invariant set]}}
$M \subseteq \mathbb{R}^d$ is \textit{invariant} for the DI \eqref{di} if for every $x \in M$ there exists 
a trajectory, $\textbf{x} \in \sum$, such that for $\textbf{x}(0) = x (\in M)$, $\textbf{x}(t) \in M$,
for all $t > 0$. Note that the definition of invariant set used in this paper, is the same as
that of positive invariant set in \cite{Benaim05} and \cite{BorkarBook}.
\item \textbf{\textbf{[Internally chain transitive set]}}
An invariant set $M \subset \mathbb{R}^{d}$ is said to be
internally chain transitive if $M$ is compact and, for every $x, y \in M$,
$\epsilon >0$ and $T > 0$, we have the following: There exist $n$ and $\Phi^{1}, \ldots, \Phi^{n}$ that
are $n$ solutions to the differential inclusion $\dot{x}(t) \in H(x(t))$,
points $x_1(=x), \ldots, x_{n+1} (=y) \in M$
and $n$ real numbers 
$t_{1}, t_{2}, \ldots, t_{n}$ greater than $T$ such that: $\Phi^i_{t_{i}}(x_i) \in N^\epsilon(x_{i+1})$ and
$\Phi^{i}_{[0, t_{i}]}(x_i) \subset M$ for $1 \le i \le n$. The sequence $(x_{1}(=x), \ldots, x_{n+1}(=y))$
is called an $(\epsilon, T)$ chain in $M$ from $x$ to $y$.
% \item \textbf{\textbf{[Chain recurrent set]}} A set is chain recurrent if the above internally chain transitive property holds 
% for all $x,\ y$ such that $x=y$.
\item \textbf{\textbf{[Attracting set \& fundamental neighborhood]}}
$A \subseteq \mathbb{R}^d$ is \textit{attracting}, if it is compact
and there exists a neighborhood $U$ such that for any $\epsilon > 0$,
$\exists \ T(\epsilon) \ge 0$ with $\Phi_{[T(\epsilon), +\infty)}(U) \subset
N^{\epsilon}(A)$. Such a $U$ is called the \textit{fundamental neighborhood} of $A$. 
The \textit{basin
of attraction } of $A$ is given by $B(A) = \{x \ | \ \underset{t \ge 0}{ \cap} \overline{\Phi_{[t, \infty)}(x)} \subset A\}$.
\item \textbf{\textbf{[Attractor set]}}
In addition to being compact, if the \textit{attracting set} is also invariant, then
it is called an \textit{attractor}.
\item \textbf{\textbf{[Inward directing sets, \cite{ramaswamy2017}]}} Given a differential inclusion $\dot{x}(t) \in H(x(t))$,
an open set $\mathcal{O}$ is said to be an inward directing set with respect to the aforementioned
differential inclusion, if $\Phi_t(x) \subseteq \mathcal{O}$, $t >0$, whenever $x \in \overline{\mathcal{O}}$.
Specifically, if $\mathcal{O}$ is inward directing, then any solution to the DI with starting point at the boundary of $\mathcal{O}$
is ``directed inwards'', into $\mathcal{O}$.
% \item \textbf{\textbf{[Lyapunov stable]}} The above set $A$ is Lyapunov stable 
% if for all $\delta > 0$, $\exists \ \epsilon > 0$ such that
% $\Phi_{[0, +\infty)}(N^\epsilon(A)) \subseteq N^\delta(A)$.
\end{itemize}}
%%%%%%%%%%%%%%%%%%%%%%%%%%%%%%%%%%%%%%%%%%%%%%%%%%%%%%%%%%%%%%%%%%%%%%%%%%%%%%%%%%%%%%%%%%%%%%%%%%%%%%%%%%%%%%
%%%%%%%%%%%%%%%%%%%%%%%%%%%%%%%%%%%%%%%%%%%%%%%%%%%%%%%%%%%%%%%%%%%%%%%%%%%%%%%%%%%%%%%%%%%%%%%%%%%%%%%%%%%%%%
%%%%%%%%%%%%%%%%%%%%%%%%%%%%%%%%%%%%%%%%%%%%%%%%%%%%%%%%%%%%%%%%%%%%%%%%%%%%%%%%%%%%%%%%%%%%%%%%%%%%%%%%%%%%%%
\section{Assumptions for convergence analysis} \label{sec_assumptions}
%As previously stated, A2VI and A2PG can be viewed as asynchronous stochastic approximations with asymptotically bounded, and possibly biased, errors.
%These errors are due to function approximations.
%Further, it is reasonable to expect the approximation errors to be bounded in an asymptotic sense, since function approximators are continuously trained.
%It is worth noting that these errors could be biased random variables (having non-zero means).
%In DeepRL, DNNs
%are used for function approximations due to their effectiveness in approximating a multitude of cost/reward functions. 
%They are used to approximate the Bellman operator, Q-factor and 
%policy gradient, among others.
%\textit{A given DNN cannot be expected to approximate a given objective function with arbitrary precision}.
%However, since a DNN is continuously trained, it is reasonable to expect the approximation errors to diminish, although they may not vanish completely.
%To account for these issues, we present a natural extension of asynchronous stochastic approximations which allows 
%for asymptotically bounded, and possibly biased, approximation errors.
%Below we state the aforementioned extension:
In this paper we are interested in the complete analysis of {\bf asynchronous SAs with non-diminishing biased additive errors}. The general iterative structure of such algorithms is given by:
\begin{equation}
 \label{asmp_aasaa}
 \begin{split}
  &x_{n+1}(i) = x_n(i) + a( \nu (n, i)) I(i \in Y_n) \\ &\left[ (\mathcal{A} f)_i (x_{n - \tau_{1 i}(n)}(1), \ldots, x_{n - \tau_{d i}(n)}(d)) + 
 M_{n+1}(i) \right], \text{ where}
 \end{split}
\end{equation}
\begin{enumerate}
 \item $x_n = (x_n(1), \ldots, x_n(d)) \in \mathbb{R}^d$, $n \ge 0$.
 \item $f: \mathbb{R}^d \to \mathbb{R}^d$ is a Lipschitz continuous objective function.
 %\item $0 \le \tau_{ij}(n) \le n$ is the delay faced by agent $j$ in receiving information from agent $i$
  %at stage $n$. In other words, we allow for unbounded delays.
  %\item $\nu(n, i) := \sum \limits_{m = 0}^n I(i \in Y_m)$ denotes the number of times that agent $i$
  %has updated it's parameter components, \textit{i.e.,} has been active until stage $n$.
 % $Y_n \subseteq \{1, 2, \ldots, d\}$ denotes the subset of agents that are active at stage $n$.
 %\item $\mathcal{A}$ is an approximation operator.
 %\item $\{a(n)\}_{n \ge 0}$ is the given step-size sequence.
 %\item $\{M_{n+1}\}_{n \ge 0}$ is a square integrable Martingale difference sequence, where $M_{n+1} = (M_{n+1}(1), \ldots, M_{n+1}(d))$.
 \end{enumerate}
 The terms $\tau_{ji}(n)$, $Y_n$, $\mathcal{A}$, $\{a(n)\}_{n \ge 0}$ and $M_{n} = (M_n(1), \ldots, M_n (d))$ are as defined for equation \eqref{intro_a2vi} in Section \ref{intro_a2}. It is worth noting that A2VI, \eqref{intro_a2vi}, and A2PG, \eqref{intro_a2pi}, are structurally identical to \eqref{asmp_aasaa}. We first present an analysis of \eqref{asmp_aasaa}. Later, this analysis is transcribed to obtain the desired theory for A2VI and A2PG. Additionally, stronger conclusions are drawn that are specific to A2VI and A2PG. Before proceeding with the analysis, the assumptions involved in the convergence analysis of (\ref{asmp_aasaa}) are listed.
 %Below we present the assumptions used in the analysis of the long-term behavior of (\ref{asmp_aasaa}). These assumptions are adaptations
%of those found in \cite{Borkar_asynchronous}.
\begin{itemize}
 \item[{\bf (A1)}] {\it $f: \mathbb{R}^d \to \mathbb{R}^d$ is a Lipschitz continuous function with Lipschitz constant $L$. Further,
 $\mathcal{A}$ is such that $\mathcal{A}f(x_n) \in f(x_n) + \overline{B}_\epsilon (0)$ for all $n \ge N$, where $N$ may be sample path dependent. Note that $\overline{B}_\epsilon(0)$ is a closed ball of radius $\epsilon$
 centered at the origin. Here $\epsilon > 0$ is a fixed upper bound on the norm of the asymptotic approximation errors.}
 \item[{\bf (A2)}] {\it The step-size sequence $\{a(n)\}_{n \ge 0}$ satisfies the following conditions:}
 \begin{itemize}
  \item[(i)] {\it $\sum \limits_{n \ge 0} a(n) = \infty$ and $\sum \limits_{n \ge 0} a(n) ^2 < \infty$.}
  \item[(ii)] {\it $\limsup \limits_{n \to \infty} \sup \limits_{y \in [x, 1]} \frac{a( \lfloor yn \rfloor)}{a(n)} < \infty$
  for $0 < x \le 1$.}
 \end{itemize}
 \item[{\bf (A3)}] {\it $\frac{n - \tau_{ij}(n)}{n} \to 1$ a.s. for every $1 \le i < j \le d$.}
  \item[{\bf (A4)}] {\it $\sup \limits_{n \ge 0} \lVert x_n \rVert < \infty$ a.s.}
  \item[{\bf (A5)}] {\it $\{M_{n+1}\}_{n \ge 0}$ is a square integrable martingale difference sequence such that}
  \begin{itemize}
   \item[] {\it $E \left[ M_{n+1}(i) \mid \mathcal{F}_n\right] = 0$.}
   \item[] {\it $E \left[ \lVert M_{n+1}(i) \rVert ^2  \mid \mathcal{F}_n\right] \le K(1 + \sup \limits_{m \le n} \lVert x_m \rVert^2 ),$  \text{ for all $n$,} where}
  \end{itemize}
  {\it $\mathcal{F}_n := \sigma \left\langle x_m, M_m, Y_m, \tau_{ij}(m); 1 \le i, j \le d, m \le n \right\rangle$,
	$1 \le i \le d$, $n \ge 0$ and $K >0$ is some fixed constant.}
\end{itemize}
We assume that all the agents are asynchronous. However, if we want the algorithm to learn effectively, then certain causal assumptions are necessary. $(A3)$ is one such assumption. Colloquially, $(A3)$ requires that the information delay between agents at time $n$ is in $o(n)$, where $o(\cdotp)$ is the standard \textit{Little-O} notation. Without loss of generality, we assume that $\tau_{ii}(n) = 0$ for all $i$ and $n$. In other words, we assume that an agent does not experience delays in accessing its own local information.

%In the following section, Section~\ref{sec_convergence}, we present the said analysis assuming stability, \textit{i.e.,} under $(A4)$.
\subsection*{Brief overview of the steps involved in our analysis}
\begin{itemize}
\item In Section \ref{sec_convergence}, convergence properties of \eqref{asmp_aasaa} are analyzed under the almost sure boundedness assumption, i.e., $(A4)$. This analysis is presented in two stages. In the first stage, presented in Section~\ref{sec_nodelay},
it is assumed that $\tau_{ij}(n) = 0$ for all $i,\ j$ and $n$, \textit{i.e.,} {\bf there are no communication delays}.
In the second stage, presented in Section~\ref{sec_delay}, the effect of communication delays is considered. Specifically, it is shown that the errors due to delayed communications do not affect the analysis in Section~\ref{sec_nodelay}.
\item In Section~\ref{sec_stability}, we swap-out $(A4)$ in favor of verifiable conditions which guarantee stability of (\ref{asmp_aasaa}). {\bf We do this by presenting assumptions that imply $(A4)$}. These assumptions are compatible with the conditions listed earlier in this section. Put together, they constitute an analytic framework for studying stability and convergence of \eqref{asmp_aasaa}.
\end{itemize}
\begin{remark}
\label{asmp_remark1}
In typical DeepRL applications, the approximation operator $\mathcal{A}$ is a DNN. The objective function $f$ is typically one among the following: value function, Q-value function, policy function and Bellman operator. The operator $\mathcal{A}$ is trained in an online manner using loss functions that reduce the ``prediction errors''. The neural network architecture is fixed by the experimenter without complete knowledge of $f$. This certainly limits how well the chosen neural network can approximate $f$. In other words, there may not exist a set of network weights such that the approximation errors are arbitrarily small. Hence, it is reasonable to merely hope that the errors do not grow over time. This is codified in $(A1)$ as $\limsup \limits_{n \to \infty} \ \lVert \mathcal{A}f(x_n) - f(x_n) \rVert \le \epsilon$ a.s. for some $\epsilon > 0$.
\end{remark}
%%%%%%%%%%%%%%%%%%%%%%%%%%%%%%%%%%%%%%%%%%%%%%%%
%\begin{remark}
%\label{asmp_remark1}
% As a consequence of $(A1)$, we get that $\sup \limits_{n \ge 0} \ \lVert \mathcal{A}f(x_n) - f(x_n) \rVert \le \epsilon$ a.s. The analysis presented in this
% paper will carry forth, verbatim, even under the weaker assumption that 
 %\begin{equation}
 %\label{asmp_remark1_eq}
   %\limsup \limits_{n \to \infty} \ \lVert \mathcal{A}f(x_n) - f(x_n) \rVert \le \epsilon\ a.s.
 %\end{equation}
 %In many RL applications, deep function approximators are typically trained in an online manner. 
 %Initially they approximate 
 %poorly, but after sufficient training, they exhibit good empirical performance. The weakening of $(A1)$, presented in this remark, is important since it
 %accounts for the aforementioned online training process. This is also an important weakening as compared to traditional literature which requires :
% \[
  %\limsup \limits_{n \to \infty} \ \lVert \mathcal{A}f(x_n) - f(x_n) \rVert = 0\ a.s.
 %\]
 %Importance of the weakening, (\ref{asmp_remark1_eq}), stems from the fact that a function approximator (eg. DNN) cannot be expected to approximate an objective function arbitrarily well.
%\end{remark}
\section{Convergence analysis} \label{sec_convergence}
We are now ready to analyze the convergence of (\ref{asmp_aasaa}) under $(A1)$-$(A5)$. We begin our analysis by making the additional assumption that there are no communication delays, i.e., $\tau_{ji}(n) = 0 \ \forall i, j, n$. This allows us to focus on the effect of asynchronicity between agents. Then, in Section \ref{sec_delay} we show that the analysis in Section \ref{sec_nodelay} is unaffected by the errors due to delayed communications.
\subsection{Analysis with no delays} \label{sec_nodelay}
Assuming $\tau_{ii}(n) = 0$ $\forall i,n$, equation \eqref{asmp_aasaa} becomes:
\begin{equation}
 \label{nodelay_sa}
 \begin{split}
 x_{n+1}(i) &= x_n (i) + a( \nu(n, i)) I(i \in Y_n) \\ &\left[(\mathcal{A}f)_i(x_n(1), \ldots, x_n(d)) + M_{n+1}(i) \right].
 \end{split}
\end{equation}
For $n \ge 0$, define $\overline{a}(n) := \max \limits_{i \in Y_n} a(\nu(n, i))$ and $q(n, i) := \frac{a(\nu(n, i))}{\overline{a}(n)} I(i \in Y_n)$. It can be shown that
$\sum \limits_{n \ge 0} \overline{a}(n)  = \infty$ and $\sum \limits_{n \ge 0} \overline{a}(n)^2  < \infty$.

Equation (\ref{nodelay_sa}) can be further rewritten as follows:
\[
\begin{split}
 x_{n+1}(i) &= x_n(i) + \overline{a}(n) q(n, i) \\ &\left[ f_i(x_n(1), \ldots, x_n(d)) + \epsilon_n(i) + M_{n+1}(i)\right],
 \end{split}
\]
where, $\epsilon_n = (\epsilon_n(1), \ldots, \epsilon_n(d))$ is the approximation error at stage $n$, \textit{i.e.,}
$\epsilon_n = \mathcal{A}f(x_n) - f(x_n)$. It follows from $(A1)$ that 
$\limsup \limits_{n \to \infty} \lVert \epsilon_n \rVert  \le \epsilon$ for a certain $\epsilon > 0$ fixed. Without loss of generality, we may say that $\lVert \epsilon_n \rVert \le \epsilon$ for all $n \ge 0$, even though we only have $\lVert \epsilon_n \rVert \le \epsilon$ for all $n \ge N$ (for sample path dependent $N$). This is because we are only interested in the asymptotic behaviour of \eqref{asmp_aasaa}. In other words, $\{x_n\}_{n \ge 0}$ and $\{x_n\}_{n \ge N}$ (subsequence starting at a sample point dependent $N$) have identical asymptotic properties.

%We use $\{ \overline{a}(n)\}_{n \ge 0}$ to define a linearly interpolated trajectory as follows. 
For $n \ge 1$, define $t(0) := 0$,
$t(n) := \sum \limits_{m = 0}^{n -1} \overline{a}(m)$.  For 
$t \in [t(n), t(n+1))$, define $\overline{x}(t) := x_n$, $\lambda(t) := diag(q(n, 1), \ldots, q(n, d))$ and $\overline{\epsilon}(t) = \epsilon_n$ for $t \in [t(n), t(n+1))$. The notation
$diag(a_1, \ldots, a_d)$ is used to denote the diagonal $d \times d$ matrix given by
\[
 \begin{bmatrix} 
    a_1 & 0 & \dots \\
    \vdots & \ddots & \\
    0 &        & a_d 
    \end{bmatrix}.
\]

In the above, $\{\overline{a}(n)\}_{n \ge 0}$ is used to divide the time-axis. The quantity $\sum \limits_{m=0}^n q(m, i)$ captures the fraction of
 time $\left(\sum \limits_{m=0}^n \overline{a}(m) \right)$ that agent $i$ is active. Thus, $q(m, \cdotp)$ captures the relative frequency of the agent updates. For more
 details the reader is referred to Borkar \cite{Borkar_asynchronous}.

Let us recall \eqref{asmp_aasaa} in the following useful form:
\begin{equation}
 \label{nodelay_useful}
 \begin{split}
 x_{n+1} =\ &x_n + \overline{a}(n)\\  &\lambda(t(n)) \left[ f(x_n) + \epsilon_n + M_{n+1}\right].
 \end{split}
\end{equation}
\begin{flushleft}
It follows from $(A4), \ (A5)$ and $\sum \limits_{n=0}^{\infty} \overline{a}(n)^2 < \infty$, that $\sum \limits_{n \ge 0} \lVert \overline{a}(n) M_{n+1} \rVert ^2 < \infty$. In other words,
the quadratic variation process associated with $\xi_n := \sum \limits_{m = 0}^n \overline{a}(m) \lambda(t(m)) M_{m+1}$, $n \ge 0$, is 
bounded almost surely. From this we may conclude that the martingale noise sequence, $\{\xi_n\}_{n \ge 0}$, is convergent almost surely. For a proof of the 
aforementioned, the reader is referred to \textit{Chapter 2} of Borkar \cite{BorkarBook}. Given the above, the following lemma is immediate.
\end{flushleft}

\begin{lemma}
 \label{nodelay_noise}
 $\lim \limits_{n \to \infty} \xi_{n} < \infty$ a.s., where \\ $\xi_n = \sum \limits_{m = 0}^n \overline{a}(m) \lambda(t(m)) M_{m+1}$.
 In other words, the martingale difference noise sequence is convergent.
\end{lemma}

%We are now ready to define a non-autonomous differential inclusion (DI), whose solutions track our algorithm.
For $s \ge 0$, define
\[
 x^s(t) := \overline{x}(s) + \int \limits_s ^{s+t} \left( \lambda(\tau) f(\overline{x}(\tau)) + \epsilon(\tau) \right) \ d \tau.
\]
Then $x^s (\cdotp)$ is a solution to the non-autonomous DI $\dot{x}(t) \in \lambda(t + s) f(x(t)) + \overline{B}_\epsilon (0)$,
with $\overline{x}(s)$ as its starting point.
% 
% \begin{proof}
%  The proof of this statement can be found in Chapter 2 of \cite{BorkarBook}. Briefly, it follows from $(A4)$ and $(A5)$ that the quadratic
%  variation process associated with $\{\xi_n\}_{n \ge 0}$ is convergent. This in turn leads to the almost sure convergence of $\{\xi_n\}_{n \ge 0}$.
% \end{proof}
It follows from the definitions of $\overline{x}(\cdotp)$, $x^s(\cdotp)$, and from Lemma~\ref{nodelay_noise} that 
\begin{equation}
 \label{nodelay_aa_eq}
\lim \limits_{s \to \infty} \sup \limits_{t \in [s, s+T]} \lVert \overline{x}(t) - x^s(t) \rVert = 0\ a.s.
 \end{equation}
 Therefore, the asymptotic behavior of \eqref{asmp_aasaa} and \eqref{nodelay_useful} can be determined by studying the family of functions given by $\left\{ x^s([0, T]) {\Large \mid} s \ge 0, \ T > 0 \right\}$.
 
For any fixed $T > 0$, the set $\{x^s([0,T]) \mid s \ge 0\}$ can be viewed as a subset of $D([0,T], \mathbb{R}^d)$, equipped
with the Skorohod topology. It follows from the Arzela-Ascoli theorem for $D([0,T], \mathbb{R}^d)$ that the aforementioned subset
is relatively compact. For details on C\`{a}dl\`{a}g spaces, Skorohod topology and the Arzela-Ascoli theorem, the reader is referred to Billingsley \cite{Billingsley}. It now follows
from (\ref{nodelay_aa_eq}) that $\{x^s([0,T]) \mid s \ge 0\}$ and $\{\overline{x}([s, s+T]) \mid s \ge 0\}$ have the same limit points in 
$D([0,T], \mathbb{R}^d)$. Hence, to find any subsequential limit of $\{\overline{x}(s + \cdotp) \mid s \ge 0\}$, we merely need to consider
the corresponding subsequence in $\{x^s([0,T]) \mid s \ge 0\}$. Finally, since $T$ is arbitrary, $\{\overline{x}(s + \cdotp) \mid s \ge 0\}$
is relatively compact in $D([0, \infty), \mathbb{R}^d)$.
\begin{lemma}
 \label{nodelay_limitset}
 Almost surely any limit point of $\{\overline{x}(s + \cdotp) \mid s \ge 0\}$ in $D([0, \infty), \mathbb{R}^d)$ is a solution to the non-autonomous $DI$
 $\dot{x}(t) \in \Lambda(t) f(x(t)) + \overline{B}_\epsilon(0)$, where $\Lambda(\cdotp)$ is a $d \times d$-dimensional diagonal matrix-valued
 measurable function with diagonal entries in $[0,1]$.
\end{lemma}
\begin{proof}
 As in the proof of \textit{Theorem 2, Chapter 7} of Borkar \cite{BorkarBook}, we view $\lambda(\cdotp)$ as an element
 of $\mathcal{V}$, where $\mathcal{V}$ is the space of measurable maps $y(\cdotp): [0, \infty) \to [0, 1]^d$ with the coarsest topology that renders continuous,
 the maps
 \[
  y(\cdotp) \to \int \limits_0 ^t \langle g(s), y(s) \rangle ds,
 \]
for all $t > 0$, $g(\cdotp) \in L_2 ([0,T], \mathbb{R}^d)$.

Define $\hat{\epsilon}_s (t) := \lambda(t) \overline{\epsilon}(t)$ for all $t \ge 0$. Since $\hat{\epsilon}_s (\cdotp)$ is measurable for every $s \ge 0$
and $\sup \limits_{s \ge 0} \lVert \hat{\epsilon}_s \rVert < \infty$, we obtain that
$\{\hat{\epsilon}_s([0, T]) \mid s \ge 0\}$ is relatively compact in $L_2([0, T], \mathbb{R}^d)$.
If necessary, by choosing
a common subsequence of $\{\hat{\epsilon}_s([0,T]) \mid s \ge 0\}$ and $\{\lambda([s, s+T]) \mid s \ge 0\}$, we can show that any limit
of $\{\overline{x}(s + \cdotp) \mid s \ge 0\}$, in $D([0, T], \mathbb{R}^d)$,
is of the form:
\[
 x(t) = x(0) + \int \limits_0 ^t \Lambda(\tau) f(x(\tau)) d\tau + \int \limits_0 ^t \epsilon(\tau) d \tau 
\]
\[
 \text{\textbf{or}}
\]
\[
  x(t) = x(0) + \int \limits_0 ^t \left[ \Lambda(\tau) f(x(\tau))  +  \epsilon(\tau) \right] d \tau,
\]

where $\epsilon(\cdotp)$ and $\Lambda(\cdotp)$ are the subsequential limits of 
$\{\hat{\epsilon}_s([0,T]) \mid s \ge 0\}$ and $\{\lambda([s, s+T]) \mid s \ge 0\}$ respectively.
Note that $\lVert \epsilon(t) \rVert \le \epsilon$, for $t \ge 0$, and that $\epsilon(\cdotp)$
is the weak limit in $L_2([0, T], \mathbb{R}^d)$, as $s \to \infty$. Also note that $\Lambda(\cdotp)$ is the limit in
$\mathcal{V}$, equipped with the coarsest topology described above.
% The details of the arguments presented above can be found in \textit{Chapter 7} of Borkar \cite{BorkarBook}.
\end{proof}
The above lemma states that algorithm \eqref{asmp_aasaa} tracks a solution to the non-autonomous DI given by
$\dot{x}(t) \in \Lambda(t)f(x(t)) + \overline{B}_\epsilon(0)$. We needed to associate a DI and not an o.d.e. since the algorithm allows for
asymptotically biased approximation errors. The non-autonomous $\Lambda(\cdotp)$ is a consequence of asynchronicity. It is clear from the above Proof of Lemma \ref{nodelay_limitset} that $\Lambda(\cdotp)$ captures the relative
update frequencies of the various agents involved in a limiting sense (proof of Lemma~\ref{nodelay_limitset}).
\subsection{Extension to account for delays} \label{sec_delay}
In this section, we show that the statement of Lemma \ref{nodelay_limitset} is true even when $\tau_{ij}(n) \neq 0$. We present additional assumptions on the step-size sequence and the delay process $\tau$. Under these additional assumptions we show that the analysis in Section \ref{sec_nodelay} remains unaffected by the errors that arise from delayed communications.
Before proceeding, we note that a methodology to deal with the effect of delays separately, was developed by Borkar in 1998, see \cite{Borkar_asynchronous}. We use similar
techniques here. In order to avoid redundancies, we only provide additional details and
a brief outline of the proof.
The reader is
referred to \cite{Borkar_asynchronous} or \cite{BorkarBook} for details. Recall the algorithm under consideration:
\begin{equation}
 \label{delay_aa_eq}
 \begin{split}
  &x_{n+1}(i) = x_n(i) + a( \nu (n, i)) I(i \in Y_n) \\ &\left[ (\mathcal{A} f)_i (x_{n - \tau_{1 i}(n)}(1), \ldots, x_{n - \tau_{d i}(n)}(d)) + 
 M_{n+1}(i) \right].
 \end{split}
\end{equation}
The previously mentioned assumptions are listed below as additional refinements in $(A2)$:
\\
 {\bf (A2)(iii)} $\sup \limits_{n \ge 0} a(n) \le 1$.\\
 {\bf (A2)(iv)} For $m \le n$, we have $a(n) \le \kappa a(m)$, where $\kappa > 0$.\\
 {\bf (A2)(v)} There exists $\eta > 0$ and a non-negative integer-valued random variable $\overline{\tau}$ such that:
	    \begin{itemize}
	     \item[(i)] $a(n) = o(n ^{- \eta})$.
	     \item[(ii)] $\overline{\tau}$ stochastically dominates all $\tau_{kl}(n)$ and satisfies
	     \[
	      E\left[ \overline{\tau}^{1/\eta} \right] < \infty.
	     \]
	    \end{itemize}
In Lemma~\ref{nodelay_limitset}, we showed that (\ref{delay_aa_eq}) tracks a solution to the non-autonomous DI:
\begin{equation}
\label{delay_NADI}
 \dot{x}(t) \in \Lambda(t) f(x(t)) + \overline{B}_\epsilon(0).
\end{equation}
In what follows we outline the proof of why (\ref{delay_aa_eq}) still tracks a solution to (\ref{delay_NADI}) even in the presence of delayed communications. Specifically, it is shown that the ``effect'' due to delays vanishes in the order of the step-sizes.

Let us consider the following quantity:
\[
\begin{split}
 &a(\nu(n, i)) I(i \in Y_n) \\ &\left| f_i(x_{n - \tau_{1i}(n)}(1), \ldots, x_{n - \tau_{di}(n)}(d)) -
 f_i(x_n(1), \ldots, x_n(d))
 \right|.
 \end{split}
\]
There are no error terms due to the approximation operator $\mathcal{A}$, since they are already considered in the analysis presented in
Section~\ref{sec_nodelay}.
Since $f$ is Lipschitz continuous, it is enough to find bounds for the terms
\[
 a(\nu(n, i)) \left| x_n(j) - x_{n - \tau_{ji}(n)}(j) \right| \text{ for every $i$ and $j$.}
\] 
Clearly, the above term is bounded by
\[
 a(\nu(n,i)) \sum \limits_{m = n - \tau_{ji}(n)}^{n - 1} \left| x_{m+1}(j) - x_m(j) \right|.
\]
Using (\ref{delay_aa_eq}) and the Lipschitz property of $f$, we get the following bound:
\[
 a(\nu(n,i)) \sum \limits_{m = n - \tau_{ji}(n)}^{n - 1} C a(m) \le C a(\nu(n,i)) \tau_{ji}(n),
\]
for some constant $C > 0$.
Our task is now reduced to showing that $a(\nu(n,i)) \tau_{ji}(n) = o(1)$, which in turn follows from 
\[
 P(\tau_{ji}(n) > n ^\eta  \ i.o.) = 0.
\]
The above equation follows from $(A2)(v)$ and the Borel-Cantelli lemma. The following theorem is an immediate consequence of the analysis
done hitherto.
\begin{theorem}
 \label{delay_main}
 Under assumptions $(A1)$-$(A5)$, the asynchronous approximation algorithm given by (\ref{asmp_aasaa}) has the same limiting set
 as the non-autonomous DI given by $\dot{x}(t) \in \Lambda(t) f(x(t)) + \overline{B}_\epsilon(0)$, where $\Lambda(t)$ is some matrix-valued measurable process.
 Further, for every $t \ge 0$, $\Lambda(t)$ is a diagonal matrix with entries in $[0,1]$.
\end{theorem}
\subsection{Balanced step-size sequences} \label{sec_balance}
A drawback in applying the above theorem to practical applications is the fact that the $DI$ (\ref{delay_NADI}) is non-autonomous. Further, $\Lambda(\cdotp)$
is not exactly known. Borkar \cite{Borkar_asynchronous} solved this problem through the use of  ``balanced step-size sequences''.  Note that a step-size sequence $\{a(\nu(n,i))\}_{n \ge 0, 1 \le i \le d}$ is balanced if there exist $a_{ij} > 0$ for every pair of $i$ and $j$ such that
\[
\lim \limits_{n \to \infty} \frac{\sum \limits_{m=0}^n a(\nu(m,i))}{\sum \limits_{m=0}^n a(\nu(m,j))} = a_{ij}.
\]
Typical diminishing step-size sequences are balanced provided all agents update their parameters with the same frequency. Since in this paper we assume that all agents are updated with the same frequency, diminishing step-size sequences are balanced.

When balanced special step-sizes are used, one obtains $\Lambda(t) = diag(1/d, \ldots, 1/d)$ for all $t \ge 0$, see Theorem 3.2 of \cite{Borkar_asynchronous} for details. 
The tracking DI, (\ref{delay_NADI}), of Theorem~\ref{delay_main} then becomes
\begin{equation}
\label{balance_DI}
\dot{x}(t) \in diag(1/d, \ldots, 1/d) f(x(t)) + \overline{B}_\epsilon (0) .
\end{equation}
As noted in \cite{abounadi}, the qualitative behaviors of $\dot{x}(t) = f(x(t))$ and $\dot{x}(t)$ $=  diag(1/d, \ldots, 1/d) f(x(t))$
are similar since they only differ in scale. Further, it follows from the upper semi-continuity of chain recurrent sets that the long-term behavior of \eqref{balance_DI}  is similar to that of 
$\dot{x}(t) =  diag(1/d, \ldots, 1/d) f(x(t))$ for small enough $\epsilon$. In other words,
the long-term behavior of (\ref{balance_DI}) approximates that of $\dot{x}(t) = f(x(t))$. 

{\it In this section, we have shown that asynchronous SAs with asymptotically bounded biased errors track a solution
to \eqref{balance_DI}, when balanced step-sizes are used.} \footnote{Recall that $\epsilon$ of (\ref{balance_DI}) is the norm-bound on the approximation errors.} 
%%%%%%%%%%%%%%%%%%%%%%%%%%%%%%%%%%%%%%%%%%%%%%%%%%%%%%%%%%%%%%%%%%%%%%%%%%%%%%%%%%%%%%%%%%%%%%%%%%%%%%%%%%%%%%%%%%%%%%%%%%%%%%%%%%%%%%%%%%%%%%%%
%%%%%%%%%%%%%%%%%%%%%%%%%%%%%%%%%%%%%%%%%%%%%%%%%%%%%%%%%%%%%%%%%%%%%%%%%%%%%%%%%%%%%%%%%%%%%%%%%%%%%%%%%%%%%%%%%%%%%%%%%%%%%%%%%%%%%%%%%%%%%%%%
%%%%%%%%%%%%%%%%%%%%%%%%%%%%%%%%%%%%%%%%%%%%%%%%%%%%%%%%%%%%%%%%%%%%%%%%%%%%%%%%%%%%%%%%%%%%%%%%%%%%%%%%%%%%%%%%%%%%%%%%%%%%%%%%%%%%%%%%%%%%%%%%
\section{Stability analysis} \label{sec_stability}
The foregoing analysis required that the iterates be bounded in an almost sure sense. This requirement is hard to ensure when function approximation is used. It is well known that unbounded approximation errors can affect the stability of the algorithm, see \cite{BertsekasBook}. {\it In this section, we present a set of sufficient conditions which ensure the following: suppose the errors are asymptotically bounded and possibly biased, then \eqref{asmp_aasaa} is stable.}

%The use of Deep Neural Networks (DNNs) for function approximations within reinforcement learning has boosted the applicability of classical reinforcement learning algorithms, to solve
%a wider variety of problems effectively. As stated earlier, one problem in using function approximations is that only suboptimal policies may be found.
%\textit{Another problem is that the resulting approximate reinforcement learning or neuro-dynamic programming algorithm can be unstable}. Before using DNNs for function approximations,
%the following question needs to be answered: what are the conditions under which a DeepRL algorithm is still stable? 
%In this section, we show that the stability of the algorithm is unaffected by function approximations, provided the approximation errors are asymptotically bounded.
%Note that these errors could be biased.
%Below are the additional assumptions on the step-size sequence
%that are standard in literature, see \cite{Borkar_asynchronous}, \cite{BorkarBook} and \cite{abounadi}.
\subsection{Stability assumptions}\label{sec_stability_assumptions}
Given $n \ge 0$ and $T > 0$, define
$m_T(n) := \max \{m \mid m \ge n,\ t(m) - t(n) \le T \}$. The goal of this section is to replace the stability assumption, $(A4)$, from Section \ref{sec_assumptions} with verifiable conditions. These will be combined with the other assumptions to provide a complete analysis of stability and convergence.
\begin{itemize}
 \item[{\bf (S1)}] 
	    \begin{itemize}
	     \item[(i)] {\it The step-size sequence is eventually decreasing, i.e., $\exists \ N$ such that $a(n)\ge a(m)$ for all $N \le n \le m$.}
	     \item[(ii)] {\it  $\lim \limits_{n \to \infty} \frac{\sum \limits_{m = 0}^{\lfloor x n \rfloor} a(m)
	     }{\sum \limits_{m = 0}^{n} a(m)} = 1$ uniformly in $x \in [y, 1]$, where $0 < y \le 1$.}
	    \end{itemize}
\item[{\bf (S2)}] \begin{itemize}
             \item[(i)] {\it $\liminf \limits_{n \to \infty} \frac{\nu(n, i)}{n+1} \ge \tau$, for some $\tau > 0$.}
             \item[(ii)] {\it $\lim \limits_{n \to \infty} \frac{\sum \limits_{m = \nu(n,i)}^{\nu(m_T(n), i)} a(m)
             }{\sum \limits_{m = \nu(n, j)}^{\nu(m_T(n), j)} a(m)}$ exists for all $i, j$.}
            \end{itemize}
\item[{\bf (S3)}] \begin{itemize}
	     \item[(i)] {\it For all $n \ge 0$, we have $\lVert M_{n+1} \rVert \le D$ a.s.}
             \item[(ii)] {\it $\lim \limits_{n \to \infty} \sum \limits_{m = n}^{m_T(n)} a(m) M_{l(m) + 1} = 0$, where $\{l(m)\}_{m \ge 0}$ is an increasing
sequence of non-negative integers satisfying $l(m) \ge m$.}
            \end{itemize}
\end{itemize}
{\it Assumption $(S3)$, is a stricter version of $(A5)$. For the analysis in this section, we use $(S3)$ instead of $(A5)$. This is only done for the sake of clarity. Later, the analysis is shown to be true merely assuming $(A5)$.}
\begin{itemize}
\item[{\bf (S4)}] {\it Associated with $\dot{x}(t) = f(x(t))$ 
is a compact set $\Lambda$, a bounded open neighborhood $\mathcal{U}$ 
$\left( \Lambda \subseteq \mathcal{U}\subseteq \mathbb{R}^d \right)$ and a
function $V: \overline{\mathcal{U}} \to \mathbb{R}^+$ such that}
	\begin{itemize}
	\item[$(i)$] {\it $\forall t \ge 0$, $\Phi_t (\mathcal{U}) \subseteq \mathcal{U}$ \textit{i.e.,} $\mathcal{U}$ is strongly positively invariant.}
	\item[$(ii)$] {\it $V^{-1} (0) = \Lambda$.}
	\item[$(iii)$] {\it $V$ is a continuous function such that for all
	$x \in \mathcal{U} \setminus \Lambda$ and $y \in \Phi_t (x)$ we have $V(x) > V(y)$, for any $t > 0$.}
	\end{itemize}
% \item[$(S4b)$] 
% Associated with $\dot{x}(t) = f(x(t))$ is a compact set $\Lambda$, a bounded open neighborhood $\mathcal{U}$ and a
% function $V: \overline{\mathcal{U}} \to \mathbb{R}^+$ such that
% 	\begin{itemize}
% 	\item[(i)] $\forall t \ge 0$ $\Phi_t (\mathcal{U}) \subseteq \mathcal{U}$ \textit{i.e.,} $\mathcal{U}$ is strongly positively invariant.
% 	\item[(ii)] $V^{-1} (0) = \Lambda$.
% 	\item[(iii)] $V$ is an upper semicontinuous function such that for all $x \in \overline{\mathcal{U}} \setminus \Lambda$ and $y \in \Phi_t (x)$ we have $V(x) > V(y)$, where $t > 0$.
% 	\item[(iv)] $V$ is bounded on $\overline{\mathcal{U}}$. In other words, $\underset{x \in \mathcal{U}}{\sup} \lVert V(x) \rVert < \infty$.
% 	\end{itemize}
\item[{\bf (S4a)}] {\it $\hat{\mathbb{A}}$ is the global attractor of $\dot{x}(t) = f(x(t))$.}
\end{itemize}
For our subsequent analysis we need that one among (S4) and (S4a) is satistfied, not both. $(S4)$ and its variant $(S4a)$ are the key to our stability analysis. %Depending on the application at hand, it may be easier to verify one variant as opposed to the other.
The two variants are overlapping yet qualitatively 
different, thereby covering a multitude of scenarios.
Note that the above Lyapunov-based stability conditions are devised based on the ones in \cite{ramaswamy2017}.
%  \item[(ii)] $V: \mathbb{R}^d \to \mathbb{R}^+$ is an
% upper semicontinuous function such that $V(x) > V(y)$ for all $x \in \mathbb{R}^d \setminus \mathcal{A}$, 
% $y \in \Phi_t (x)$ and $t > 0$.
% \item[(iii)] $V(x) \ge V(y)$ for all $x \in \mathcal{A}$, $y \in \Phi_t (x)$ and $t > 0$.

If $(S4)$ is satisfied, then
\textit{Proposition 3.25} of Bena\"{i}m, Hofbauer and Sorin \cite{Benaim05} 
implies that $\dot{x}(t) = f(x(t))$ has an
attractor set $\hat{\mathbb{A}} \subseteq \Lambda$.  It also implies that $V^{-1}([0, r])$
is a fundamental neighborhood of $\hat{\mathbb{A}}$, for small values of $r$. On the other hand, if  $(S4a)$ is satisfied, then
any compact neighborhood of $\hat{\mathbb{A}}$ is a fundamental neighborhood of it. \textbf{\textit{In both cases
we can associate an attractor, $\hat{\mathbb{A}}$, and a fundamental neighborhood, $\overline{\mathcal{N}}$, to $\dot{x}(t) = f(x(t))$}}.

Given $\delta > 0$, $\exists \ \epsilon(\delta) > 0$ such that 
$\dot{x}(t) \in f(x(t)) + \overline{B}_{\epsilon(\delta)}(0)$ has an attractor 
$\mathbb{A} \subseteq N^\delta (\hat{\mathbb{A}})$ 
\textit{with fundamental neighborhood $\overline{\mathcal{N}}$}.
%For a definition of $N^\delta(\cdotp)$ see Section~\ref{sec_definitions}.
This is a consequence
of the upper semicontinuity of attractor sets, see \cite{aubin2012differential} or \cite{Benaim05} for details. It will be shown that \eqref{asmp_aasaa} converges to a neighborhood of a local/global attractor of $\dot{x}(t) = f(x(t))$, such as $\hat{\mathbb{A}}$. Further, this neighborhood depends on the approximation errors. Typically, the experimenter decides on the expected accuracy of the algorithm. This accuracy is quantified by $\delta$. Once this accuracy is fixed, the function approximator (DNN) is trained to control the asymptotic errors to $\epsilon(\delta)$. Then one can show that \eqref{asmp_aasaa} converges to $N^{\delta}(\hat{\mathbb{A}})$.

%We proceed by assuming that a $\delta$ was chosen based on the problem at hand. This automatically imposes
%a norm-bound of $\epsilon$ on the approximation errors, asymptotically speaking. This is because (\ref{asmp_aasaa})
%tracks a solution to $\dot{x}(t) \in f(x(t)) + \overline{B}_\epsilon(0)$ and $\epsilon$ is fixed as a consequence of choosing $\delta$.

Before we proceed, we associate the following Lyapunov function to $\dot{x}(t) \in f(x(t)) + \overline{B}_\epsilon(0)$:
$\tilde{V}: \overline{\mathcal{N}} \to \mathbb{R}_+$ such that
$
 \tilde{V}(x) := \max \left\{ d(y, \mathbb{A}) g(t) \mid y \in \Phi_t(x), t \ge 0 \right\}
$ and
$c \le g(t) \le d$ is a strictly increasing function with $c > 0$. Since $\overline{\mathcal{N}}$ is a fundamental neighborhood of
$\mathbb{A}$, it follows that $\sup \limits_{x \in \overline{\mathcal{N}}} \tilde{V}(x) < \infty$.

The stability analysis requires choosing two bounded open sets, say $\mathcal{B}$ and $\mathcal{C}$, such that $\mathcal{C}$ is inward directing and $\mathbb{A} \subset \mathcal{B} \subset \overline{\mathcal{B}} \subset \mathcal{C}$. Recall that $\mathbb{A}$ is an attractor of $\dot{x}(t) \in f(x(t)) + \overline{B}_\epsilon(0)$
obtained from the definition of $\hat{\mathbb{A}}$ $\left(\text{see }(S4a)\right)$. First, we choose $\mathcal{V}_r$ as $\mathcal{C}$ such that $\overline{\mathcal{V}_r} \subset \mathcal{U}$. This is possible for small values of $r$. 
Next, we choose an open $\mathcal{B}$ such
that $\mathbb{A} \subset \mathcal{B} \subset \overline{\mathcal{B}} \subset \mathcal{C}$. 
This is possible since $\Lambda$ is compact and $\mathcal{C}$ is open.	

The following two propositions are necessary for our stability analysis. The reader is referred to \cite{ramaswamy2017} for their proofs.
\begin{proposition}%[Proposition $1$ of \cite{ramaswamy2017}]
 \label{stability_prop1}
 For any $r < \sup \limits_{u \in \mathcal{N}} \tilde{V}(u)$, the set $\mathcal{V}_r := \{x \mid \tilde{V}(x) < r\}$ is open
 relative to $\overline{\mathcal{N}}$. Further, $\overline{\mathcal{V}}_r = \{x \mid \tilde{V}(x) \le r \}$.
\end{proposition}
%%%%%%%%%%%
\begin{proposition}%[Proposition $2$ of \cite{ramaswamy2017}]
 \label{stability_prop2}
 $\mathcal{C}$ is an inward directing set associated with $\dot{x}(t) \in f(x(t)) + \overline{B}_\epsilon(0)$.
\end{proposition}
Traditionally, the stability of algorithms such as \eqref{asmp_aasaa} is ensured by projecting the iterates onto a compact set at every stage. If this set is not carefully chosen, then the algorithm may not converge, or converge to an undesirable set. Using the previously constructed $\mathcal{B}$ and $\mathcal{C}$, we obtain the following {\bf projective counterpart} of \eqref{asmp_aasaa}:
\begin{equation}
 \label{sa_proj_new}
 \hat{x}_{n+1} = z_n \text{ such that } z_n \in \pro _{\mathcal{B,C}}(\tilde{x}_n).\\
\end{equation}
In the above equation, $\hat{x}_{0} = z_0$ such that $z_0 \in \pro _{\mathcal{B,C}}(x_0)$ and $\tilde{x}_{n}(i) = x_n(i) + a( \nu (n, i)) I(i \in Y_n) \left[ (\mathcal{A} f)_i (x_{n - \tau_{1 i}(n)}(1), \ldots, x_{n - \tau_{d i}(n)}(d)) + 
 M_{n+1}(i) \right]$.
From the above set of equations, it is clear that the projective iterates $\{\hat{x}_n\}_{n \ge 0} \subseteq \mathcal{C}$. Since $\mathcal{C}$ is bounded by construction (see above), $\sup \limits_{n \ge 0} \lVert \hat{x}_n \rVert < \infty$ a.s.

The realization of the projective scheme given by \eqref{sa_proj_new} depends on finding sets $\mathcal{B}$ and $\mathcal{C}$. However, from the previous discussion it is clear that $\mathcal{B}$ and $\mathcal{C}$ surely exist but may be unknown. In other words, \eqref{sa_proj_new} is only used in the stability analysis of \eqref{asmp_aasaa} and need not be realizable. Below we state our final stability assumption.
\begin{itemize}
\item[{\bf (S5)}] {\it $\sup \limits_{n \ge N} \lVert x_n - \tilde{x}_n \rVert < \infty$ a.s. for a sample path dependent $N$.}
\end{itemize}
%%%%%%%%%%%%%%%%%%%%%%%%%%%%%%%%%%%%%%%%%%%%%%%%%%%%%%%%%%%%%%%%%%%%%%%%%%%%%%%%%%%%%%%%%%%%%%%%%%%%%%%%%%%%%%%%%%%%%%%%%%%%%%%%%%%%%%%%%%%%%%%%%%
\subsection{The projective counterpart of \eqref{asmp_aasaa}} \label{sec_projective}
In this section we begin the study of \eqref{asmp_aasaa} by analysing its projective counterpart given by \eqref{sa_proj_new}. In the previous section we used $
\hat{x}_n$ to represent the projected iterates, to distinguish from the original iterates generated by \eqref{asmp_aasaa}. {\bf In this and the following couple of sections, we simply use $x_n$ for the projected iterates, instead of $\hat{x}_n$s, to reduce clutter.} We consider the following useful equivalent to \eqref{sa_proj_new}:
\begin{equation} \label{projective_proj1}
\begin{split}
\tilde{x}(n+1) &= x_n + D_n \left[\mathcal{A}f(x_n) + M_{n+1} \right], \\
x_{n+1} &= z_n, \text{ where } z_n \in \pro_{\mathcal{B,C}}(\tilde{x}_{n+1}),
\end{split}
\end{equation}
where $D_n = diag\left(a(\nu(n,1)I(1 \in Y_n), \ldots, a(\nu(n, d))I(d \in Y_n)) \right)$. 
Note that (\ref{projective_proj1}) does not account for delayed communications. However, the modifications neccesary to account delays were presented in Section~\ref{sec_delay} and can be applied here as well. Without loss of generality, assume that $Y_n$ has cardinality one for all $n \ge 0$. This is a useful trick from Abounadi et al., \cite{abounadi}. There is no loss of generality because the agents being updated at time $n$ can be viewed as being updated serially.
In other words, $Y_n = \{ \phi_n\}$ with $\phi_n \in \{1,\ldots, d\}$
for all $n \ge 0$. We may rewrite (\ref{projective_proj1}) as:
\begin{equation}
 \label{projective_proj}
 x_{n+1} = x_n + D_n \left[ f(x_n) + \epsilon_n + M_{n+1}\right] + g_n,
\end{equation}
\begin{flushleft}
 where 
$
g_n = \pro_{\mathcal{B,C}} \left( D_n \left[ f(x_n) + \epsilon_n + M_{n+1}\right] \right)
- \left( D_n \left[ f(x_n) + \epsilon_n + M_{n+1}\right] \right).$ Define $\mu_n := \left[I(\phi_n = 1), \ldots, I(\phi_n = d) \right]$,
$\overline{a}(n, i) := a(\nu(n,i))$, $\hat{a}(n) := \overline{a}(n, \phi_n)$, $t(0) := 0$ and $t(n) := \sum \limits_{m=0}^{n-1} \hat{a}(m)$
for $n \ge 1$.
\end{flushleft} 

Below we define the trajectories necessary for our analysis. It is suggested that the reader skip these definitions and refer back when required.
\begin{enumerate}
 \item[] $\mu(t) := \mu_n$ for $t \in [t(n), t(n+1))$,
 \item[] $D_c(t) := D_n$ for $t \in [t(n), t(n+1))$,
 \item[] $X_c(t) := x_n$ for $t \in [t(n), t(n+1))$,
 \item[] $Y_c(t) := \mathcal{A}f(x_n)$ for $t \in [t(n), t(n+1))$,
 \item[] $G_c(t) := \sum \limits_{m=0}^{n-1} g_m$ for $t \in [t(n), t(n+1))$,
 \item[] $\epsilon_c(t) := \mu_n \epsilon_n$ for $t \in [t(n), t(n+1))$,
 \item[] $X_l(t) := \begin{cases}
	x_n \text{ for $t = t(n)$}  \\
	\left(1 - \frac{t - t_n}{\hat{a}(n)} \right) X_l(t(n)) + \\ \left(\frac{t - t_n}{\hat{a}(n)} \right) X_l(t(n+1))  
	\text{ for $t \in [t(n), t(n+1)),$}
	\end{cases}$
  \item[] $W_l(t) := \begin{cases}
	\sum \limits_{m=0}^{n-1} D_m M_{m+1} \text{ for $t = t(n)$}  \\
	\left(1 - \frac{t - t_n}{\hat{a}(n)} \right) W_l(t(n)) +\\ \left(\frac{t - t_n}{\hat{a}(n)} \right) W_l(t(n+1))  
	\text{ for $t \in [t(n), t(n+1)).$}
	\end{cases}$
\end{enumerate}
We also define the following left-shifted trajectories:
\begin{enumerate}
 \item[] $X_l^n(t) := X_l(t + t(n))$,
 \item[] $X_c^n(t) := X_c(t + t(n))$,
 \item[] $Y_c^n(t) := Y_c(t + t(n))$,
 \item[] $W_l^n(t) := W_l(t + t(n))$,
 \item[] $G_c^n(t) := G_c(t + t(n)) - G_c(t(n))$,
 \item[] $\epsilon_c^n(t) := \epsilon_c (t + t(n))$,
 \item[] $\mu^n(t) := \mu(t + t(n))$,
 \item[] $D_c ^n(t) := D_c(t + t(n))$.
\end{enumerate}
Note that $D^n_c(t) \le 1$ and $\lVert \epsilon_c^n(t) \rVert \le \epsilon$ for all $t \ge 0$ and $n \ge 0$. 
Hence $\{D_c^n([0,T]) \mid n \ge 0\}$ and $\{\epsilon_c^n([0,T]) \mid n \ge 0\}$ are relatively compact in $\mathbb{L}_2([0,T], \mathbb{R}^d)$.

One may view $\{X_l^n([0,T]) \mid n \ge 0\}$ and $\{G_c^n([0,T]) \mid n \ge 0\}$ as subsets of
$D([0,T], \mathbb{R}^d)$ equipped with the Skorohod topology. In Lemma~\ref{projective_rc}, we show that the aforementioned
families of trajectories are relatively compact. As in \textit{Lemma 2} of \cite{ramaswamy2017} we only need to show
that these families are point-wise bounded and that any two discontinuities are separated by at least $\Delta > 0$.
%%%%%%%%%%%%%%%%%%%%%%%%%%%%%%%%%%%%%%%%%%%%%%%%%%%%%%%%%%%%%%%%%%%%%%%%%%%%%%%%%%%%%%%%%%%%%%%%%%%%%%%%%%%%%%%%%%%%%%%%%%%%%%%%%%%%%%%%%%%%%%%%
%%%%%%%%%%%%%%%%%%%%%%%%%%%%%%%%%%%%%%%%%%%%%%%%%%%%%%%%%%%%%%%%%%%%%%%%%%%%%%%%%%%%%%%%%%%%%%%%%%%%%%%%%%%%%%%%%%%%%%%%%%%%%%%%%%%%%%%%%%%%%%%%
%%%%%%%%%%%%%%%%%%%%%%%%%%%%%%%%%%%%%%%%%%%%%%%%%%%%%%%%%%%%%%%%%%%%%%%%%%%%%%%%%%%%%%%%%%%%%%%%%%%%%%%%%%%%%%%%%%%%%%%%%%%%%%%%%%%%%%%%%%%%%%%%
\begin{lemma}
 \label{projective_rc}
 $\{X_l^n([0,T]) \mid n \ge 0\}$ and $\{G_c^n([0,T]) \mid n \ge 0\}$ are relatively compact in 
 $D([0,T], \mathbb{R}^d)$, equipped with the Skorohod topology.
\end{lemma}
\begin{proof}
 As stated earlier, we only need to show that the aforementioned families of trajectories are point-wise bounded and that any two discontinuities
 are separated by at least $\Delta > 0$. From $(S3)(i)$ we have that $\lVert M_{n+1} \rVert \le D$ a.s. for all $n \ge 0$. Since $f$ is Lipschitz
 continuous, $F(x) := f(x) + \overline{B}_\epsilon (0)$ is Marchaud. 
 Clearly, $\mathcal{A}f(x_n) \in F(x_n)$ for all $n \ge 0$.
 
 We have the following:
 \[
  \sup \limits_{x \in \overline{C}, y \in F(x)} \lVert y \rVert \le C_1 \text{ for some }C_1 > 0 
 \]
\[
  \implies \sup \limits_{n \ge 0}\ \lVert \tilde{x}_{n + 1} - x_n \rVert \le \left( \sup \limits_{n \ge 0} a(n) \right) (C_1 + D)
\]
\[
 \implies \sup \limits_{n \ge 0}\ \lVert g(n) \rVert \le \sup \limits_{n \ge 0} \left(\lVert \tilde{x}_{n + 1} - x_n \rVert + d(x_n, \mathcal{B}) \right)
 \le C_2
\]
for some $0 < C_2 < \infty$ that is independent of $n$. 

Now that the point-wise boundedness property has been proven, it is left to show
that any two discontinuities are separated by some $\Delta >0$. Using arguments similar to the ones found in the proof of 
\textit{Lemma 2} in \cite{ramaswamy2017}, we can show that such a constant is given by
\[
  \Delta = \frac{d }{2\left(D + \sup \limits_{x \in \overline{C}, y \in F(x)} \lVert y \rVert \right)},
\]
where $d$ is the number of agents in the multi-agent system at hand.
\end{proof}
\subsection{Overview of the strategy involved in stability analysis}
Note that $T$ in the above lemma is arbitrary. Hence, the sets $\{X_l^n([0,\infty)) \mid n \ge 0\}$ and $\{G_c^n([0,\infty)) \mid n \ge 0\}$
are also relatively compact in $D([0,\infty), \mathbb{R}^d)$. It follows from $(S3)$ that $\{W_l ^n([0,\infty)) \mid n \ge 0 \}$ is relatively
compact in $D([0,\infty)), \mathbb{R}^d)$, and that all limits equal the constant-$0$-function. If we consider a subsequence $\{m(n)\} \subset \{n\}$ such that $M_{m(n)}$ is convergent, then
$X_l^{m(n)}([0,T])$ and
$
 X_l ^{m(n)}(0) + \int \limits_{0}^T \left( \mu ^{m(n)} (s) f(X_c ^{m(n)} (s)) + \epsilon_c ^{m(n)}(s) \right) ds  + G_c ^{m(n)}(T)$
have identical limits.

Consider a subsequence $\{m(n)\}_{n \ge 0} \subseteq \mathbb{N}$ such that $\{\epsilon_c^{m(n)}([0,T]) \mid n \ge 0\}$
is weakly convergent in $\mathbb{L}_2([0,T])$, and such that $\{X_l^{m(n)}([0,T]) \mid n \ge 0\}$ and $\{G_c ^{m(n)}([0,T]) \mid n \ge 0\}$
are convergent in $D([0,T], \mathbb{R}^d)$. In addition, this subsequence satisfies the condition that $g_{m(n) - 1} = 0$
for all $n \ge 0$. Now, let us suppose that the limit of $\{G_c^{m(n)}([0,T])\}_{n \ge 0}$ is the constant-$0$-function. Using arguments from Section~\ref{sec_nodelay}, we show that the limit of $\{X_l^{m(n)}([0,T]) \mid n \ge 0\}$ is given by:
\[
 X(0) + \int \limits_0^t \left( \lambda(s) f(X(s)) + \epsilon(s) \right) ds,
\]
such that $X(0) \in \overline{\mathcal{C}}$.
Hence, the projective scheme (\ref{projective_proj1}) tracks
a solution to $\dot{x}(t) \in \lambda(t) f(x(t)) + \overline{B}_\epsilon (0)$, where $\lambda(\cdotp)$ is some measurable matrix-valued process
with only diagonal entries. If balanced step-sizes (see \textit{Theorem 3.2} of Borkar \cite{Borkar_asynchronous}) are used, then
(\ref{projective_proj1}) tracks a solution to $\dot{x}(t) \in 1/d\ f(x(t)) + \overline{B}_\epsilon(0)$. The asymptotic behaviors of 
$\dot{x}(t) = f(x(t))$ and $\dot{x}(t) = (1/d)\ f(x(t))$ are similar, \textit{i.e.,} any solution trajectory of both o.d.e.'s, with starting points in
$\overline{\mathcal{C}}$, will converge to the attractor $\hat{\mathbb{A}}$. Consequently, any solution trajectory of 
$\dot{x}(t) \in (1/d)\ f(x(t)) + \overline{B}_\epsilon (0)$ converges to $\mathbb{A}$, provided the starting point is inside $\mathcal{C}$.
Recall that $\mathbb{A}$ is an attractor of $\dot{x}(t) \in f(x(t)) + \overline{B}_{\epsilon}(0)$ 
with fundamental neighborhood $\overline{\mathcal{N}}$ such that $\mathcal{C} \subset \overline{\mathcal{N}}$.
In other words,
the projective scheme (\ref{projective_proj1}) converges to $\mathbb{A}$ almost surely. Stability of
the algorithm under consideration, (\ref{asmp_aasaa}), follows from $(S5)$. 

\noindent
To summarize, 
there are two important steps in proving stability: \\ 
(Step-1) Any limit of $\{X_l^n([0,T])\}_{n \ge 0}$
is of the form
\[
X(t) = X(0) + \int \limits_0 ^t \left(\mu^* f(X(s)) + \epsilon(s) \right) ds + G(t) \text{ for } t \in [0,T],
\]
where $\mu ^* = diag(1/d, \ldots, 1/d)$ and $X(0) \in \overline{C}$.\\
(Step-2) Show that any limit of $\{G_c ^{m(n)}([0,T]) \mid n \ge 0\}$ is the constant $0$ function, provided $g_{m(n) - 1} = 0$ for all $n \ge 0$.

\subsection{Stability theorem}
Define $K := \{n \mid g_{n - 1 } = 0\}$. The premise of the following two lemmas is that balanced step-sizes
(of \textit{Theorem 3.2}, \cite{Borkar_asynchronous}) are used.
\begin{lemma}
 \label{projective_gis0_1}
 Without loss of generality, let $\{\epsilon_c ^n ([0,T])\}_{n \in K}$ be (weakly) convergent in 
 $\mathbb{L}_2 ([0,T], \mathbb{R}^d)$, with weak limit $\epsilon(\cdotp)$. Also let $\{X_l^n([0,T])\}_{n \in K}$ and $\{G_c^n([0,T])\}_{n \in K}$
 be convergent in $D([0,T], \mathbb{R}^d)$ as $n \to \infty$, with limits $X(\cdotp)$
 and $G(\cdotp)$ respectively. Then, for $t \in [0,T]$
  \begin{equation}
  \label{gis0_1_eq}
   \begin{split}
   X_l^n(t) \to X(0) + \int \limits_0 ^t \left(\mu^* f(X(s)) + \epsilon(s) \right) ds + G(t).
    \end{split}
  \end{equation}
\end{lemma}
\begin{proof}
%  We have the following: \\
%  (i) Let $G(\cdotp)$ be the limit of $\{G_c^n([0,T])\}_{n \in K}$
%  and $X(\cdotp)$ the limit of $\{X_l^n([0,T])\}_{n \in K}$, in $D([0,T], \mathbb{R}^d)$ equipped
%  with the Skorohod topology.\\
%  (ii) Let $\epsilon(\cdotp)$ be the weak limit of $\epsilon_c^n(\cdotp)$ in 
%  $\mathbb{L}_2([0,T], \mathbb{R}^d)$.
 Since $X_c^n(t) \to X(t)$ for $t \in [0,T]$, we get
 \[
  \int \limits_0^t \mu^* f(X_c ^n(s))  ds \to \int \limits_0 ^t \mu^* f(X(s)) ds.
 \]
Note that we have
\[
\begin{split}
 X_l^n(t) =  &X_l^n(0) + \int \limits_0^t diag(\mu^n_c(s)) f(X_c^n(s)) ds\ +\\ &W_l^n(t) + G_c^n(t) + 
 \int_0 ^t \epsilon_c ^n(s) ds.
 \end{split}
\]
Adding and subtracting $\int \limits_0^t \mu^* f(X_c^n(s)) ds$ in the above equation, we obtain:
\begin{equation}
\label{projective_gis0_3}
\begin{split}
 X_l^n(t) =  &X_l^n(0) + \int \limits_0^t \mu^* f(X_c^n(s)) ds\ + \\ &W_l^n(t) + G_c^n(t) + 
 \int_0 ^t \epsilon_c ^n(s) ds + \eta_n(t),
 \end{split}
\end{equation}
\begin{flushleft}
where $\eta_n(t) = \int \limits_0^t diag(\mu^n_c(s)) f(X_c^n(s)) ds - \int \limits_0^t \mu^* f(X_c^n(s)) ds$. 
\end{flushleft}
From Assumption $(S3)$ it follows that 
$\lim \limits_{n \to \infty} \sup \limits_{t \in [0,T]} \lVert W_l^n(t) \rVert = 0$. Suppose we show that
$\lim \limits_{n \to \infty} \sup \limits_{t \in [0,T]} \lVert \eta_n(t) \rVert = 0$, then we may use
the previously mentioned observations to conclude that
(\ref{projective_gis0_3}) converges to
\[
\begin{split}
 X(t) = X(0) &+ \int \limits_0^t \mu^* f(X(s)) ds + G(t) + \\
 &\int_0 ^t \epsilon(s) ds \text{ as $n \to \infty$}.
 \end{split}
\]
Recall that $\epsilon(\cdotp)$ is the weak limit of $\{\epsilon_c ^n ([0,T])\}_{n \in K}$.
Thus, it is left to show that $\lim \limits_{n \to \infty} \sup \limits_{t \in [0,T]} \lVert \eta_n(t) \rVert = 0$.
The proof of this is along the lines of the proof of \textit{Lemma 3.5} in Abounadi et al., \cite{abounadi}.
\end{proof}
%%%%%%%%%%%%%%%%%%%%%%%%%%%%%%%%%%%%%%%%%%%%%%%%%%%%%%%%%%%%%%%%%%%%%%%%%%%%%%%%%%%%%%%%%%%%%%%%%%%%%%%%%
%%%%%%%%%%%%%%%%%%%%%%%%%%%%%%%%%%%%%%%%%%%%%%%%%%%%%%%%%%%%%%%%%%%%%%%%%%%%%%%%%%%%%%%%%%%%%%%%%%%%%%%%%
%%%%%%%%%%%%%%%%%%%%%%%%%%%%%%%%%%%%%%%%%%%%%%%%%%%%%%%%%%%%%%%%%%%%%%%%%%%%%%%%%%%%%%%%%%%%%%%%%%%%%%%%%
\begin{lemma}
 \label{projective_gis0_2}
The $G(\cdotp)$ of Lemma~\ref{projective_gis0_1} is the constant $0$ function. As a consequence
the projective scheme (\ref{projective_proj}) converges to $\mathbb{A}$.
\end{lemma}
\begin{proof}
 For a proof of this lemma the reader is referred to the proof of \textit{Lemma 3} of \cite{ramaswamy2017}.
\end{proof}
We are now ready to state the other main result of this paper. Again, we assume that balanced step-sizes are used.
\begin{theorem}
 \label{projective_main}
 Under $(A1)$-$(A3)$ and $(S1)$-$(S5)$, the iteration given by (\ref{asmp_aasaa}) is stable ($\sup \limits_{n \ge 0}\ \lVert x_n \rVert < \infty$ a.s.) and converges to a closed connected internally chain transitive invariant set
 associated with $\dot{x}(t) \in \mu^* f(x(t)) + \overline{B}_\epsilon(0)$.
\end{theorem}
\begin{proof}
The reader may recall that $\mu^* = diag(1/d, \ldots, 1/d)$.
 It follows from Lemma~\ref{projective_gis0_2} that the associated projective iterates, 
 say $\{\hat{x}_n\}_{n \ge 0}$, corresponding to $\{x_n\}_{n \ge 0}$ converge to $\mathbb{A}$. In other words,
 there exists $N$, possibly sample path dependent, such that $\hat{x}_n \in \mathcal{C}$ for $n \ge N$.
 It follows from $(A5)$ that $\sup \limits_{n \ge N} \lVert x_n \rVert < \infty$ a.s. 
 
 The second part of the statement directly follows from Theorem~\ref{delay_main}.
\end{proof}

\subsection{Stability assuming $(A5)$ instead of $(S3)$}
The statement of Theorem \ref{projective_main} is true when the weaker $(A5)$ is assumed instead of the stricter $(S3)$. 
The details involved (in a related setup) can be found in \textit{Section 6} of \cite{ramaswamy2017}.
We merely present the steps involved without any proofs and refer the reader to 
\textit{Section 6} of \cite{ramaswamy2017} for details. 

The purpose of $(S3)$ is to show that any two discontinuities of $\{X_l^n([0,T]) \mid n \ge 0\}$
and $\{G_c^n([0,T]) \mid n \ge 0\}$ are at least $\Delta$ apart.
An important step in proving the aforementioned
claim with $(A5)$ replacing $(S3)$ is the following auxiliary lemma.
\begin{lemma}[\textit{Lemma 5}, \cite{ramaswamy2017}] \label{projective_gn_1}
Let $\{t_{m(n)}, t_{l(n)}\}_{n \ge 0}$ be such that $t_{l(n)} > t_{m(n)}$, $t_{m(n+1)} > t_{l(n)}$ 
and $\underset{n \to \infty}{\lim} \left(t_{l(n)} - t_{m(n)} \right) = 0$. Fix an arbitrary $c > 0$ and consider the following:
\[
\psi_n := \left \lVert \sum \limits_{i = m(n)}^{l(n)-1} a(i) M_{i+1} \right \rVert.
\]
Then $P \left( \{\psi_n > c\}\  i.o. \right) = 0$ within the context of 
the projective scheme given by (\ref{projective_proj}).
\end{lemma}
Colloquially, Lemma~\ref{projective_gn_1} states the following: After a lapse of considerable time, there are no significant contributions to jumps in $X_l^n(\cdotp)$
or $G_c^n(\cdotp)$ from
the Martingale difference noise sequence within shrinking time intervals.
If we are unable to find a separating $\Delta$, then it can be shown that Lemma~\ref{projective_gn_1}
is contradicted. Therefore, Theorem~\ref{projective_main} is true under the standard, weak assumption
on noise imposed by $(A5)$. 
As a consequence, the following modification of Theorem~\ref{projective_main} is immediate.
\begin{theorem}
 \label{projective_main_1}
 Under $(A1)$-$(A3)$, $(A5)$ and $(S1), (S2), (S4)$ and $(S5)$, the iteration given by (\ref{asmp_aasaa}) is bounded
 almost surely (stable) and converges to a closed connected internally chain transitive invariant set
 associated with $\dot{x}(t) \in \mu^* f(x(t)) + \overline{B}_\epsilon(0)$.
\end{theorem}
%%%%%%%%%%%%%%%%%%%%%%%%%%%%%%%%%%%%%%%%%%%%%%%%%%%%%%%%%%%%%%%%%%%%%%%%%%%%%%%%%%%%%%%%%%%%%%%%%%%%%%%%%%%%%%%%%%%%%%%%%%%%%%%%%%%%%%%%%%%%%%%%
%%%%%%%%%%%%%%%%%%%%%%%%%%%%%%%%%%%%%%%%%%%%%%%%%%%%%%%%%%%%%%%%%%%%%%%%%%%%%%%%%%%%%%%%%%%%%%%%%%%%%%%%%%%%%%%%%%%%%%%%%%%%%%%%%%%%%%%%%%%%%%%%
%%%%%%%%%%%%%%%%%%%%%%%%%%%%%%%%%%%%%%%%%%%%%%%%%%%%%%%%%%%%%%%%%%%%%%%%%%%%%%%%%%%%%%%%%%%%%%%%%%%%%%%%%%%%%%%%%%%%%%%%%%%%%%%%%%%%%%%%%%%%%%%%
%%%%%%%%%%%%%%%%%%%%%%%%%%%%%%%%%%%%%%%%%%%%%%%%%%%%%%%%%%%%%%%%%%%%%%%%%%%%%%%%%%%%%%%%%%%%%%%%%%%%%%%%%%%%%%%%%%%%%%%%%%%%%%%%%%%%%%%%%%%%%%%%
\section{Applications} \label{sec_applications}
Value and policy iterations are popular reinforcement learning algorithms that are at once effective and easy to implement. As explained in Section \ref{intro_a2}, value and policy iterations are coupled with function approximation to counter Bellman's curse of dimensionality, arising in large-scale (continuous state and action spaces) learning and control problems. In related work documented in \cite{ramaswamy2017}, value iteration with function approximation is analyzed. However, it does not consider the multi-agent setting. Abounadi et. al. \cite{abounadi} analyzed the asynchronous version of Q-learning, but without function approximation.

In this section, the theory hitherto developed is used to present a complete analysis of A2VI and A2PG. A2VI and A2PG are the multi-agent counterparts of value iteration and policy gradient scheme, respectively, that also account for function approximation.
\subsection{Asynchronous approximate value iteration (A2VI)} \label{sec_a2vi}
Recall that we are interested in the recursion:
\begin{equation}
 \label{a2vi_a2vi}
 \begin{split}
  &J_{n+1}(i) = J_n(i) + a( \nu (n, i)) I(i \in Y_n) \\ &\left[ (\mathcal{A} T)_i (J_{n - \tau_{1 i}(n)}, \ldots, J_{n - \tau_{d i}(n)}) + 
 M_{n+1}(i) \right], \text{ where}
 \end{split}
\end{equation}
\begin{enumerate}
 \item $T$ is the Bellman operator,
 \item Let $\epsilon_n = (\mathcal{A}T)J_n - TJ_n$ be the approximation error at stage $n$. 
 The approximation operator $\mathcal{A}$ could be a deep neural network, or some other function approximator.
%  \item The remaining terms are as defined before.
\end{enumerate}
\begin{remark}
We do not distinguish between stochastic shortest path and 
infinite horizon discounted cost problems. The definition of the
Bellman operator $T$ changes appropriately based on the choice of problem \cite{BertsekasBook}. 
\end{remark}
The following assumptions are natural:
\begin{itemize}
 \item[(AV1)] The Bellman operator $T$ is contractive with respect to some weighted max-norm, 
 $\lVert \cdotp \rVert_\nu$, \textit{i.e.,} $\lVert Tx - Ty \rVert_\nu \le \alpha \lVert x - y \rVert_\nu$
 for some $0 < \alpha < 1$.
 \item[(AV2)] $T$ has a unique fixed point $J^*$ and $J^*$ is the 
 unique globally asymptotically stable equilibrium point of $\dot{J}(t) = TJ(t) - J(t)$.
 \item[(AV3)] $\limsup \limits_{n \to \infty}\ \lVert \epsilon_n \rVert_\nu \le \epsilon$ for some fixed $\epsilon > 0$.
\end{itemize}
 Given $x \in \mathbb{R}^d$ we make the following simple observations:\\
(i) $\lVert x  \rVert_\nu \le \frac{1}{\min \limits_{i} \nu_i} \lVert x \rVert$.\\
(ii) $\lVert x \rVert \le \frac{d}{\min \limits_{i} \nu_i} \lVert x \rVert_\nu$. \\
The following claim is an immediate consequence of these observations.
\begin{claim}
 \label{a2vi_lipschitz}
 $T$ is Lipschitz continuous with some Lipschitz constant $0 < L < \infty$.
\end{claim}

The only difference between (\ref{a2vi_a2vi}) and (\ref{asmp_aasaa}) is that in (\ref{a2vi_a2vi}) the approximation errors
are bounded in the weighted max-norm sense. It is worth noting that the errors could be more generally bounded in the
weighted p-norm ($\lVert \cdotp \rVert_{\omega, p}$) sense. However it can be easily shown that
$C_l \lVert x \rVert_\nu \le \lVert x \rVert_{\omega, p} \le C_u \lVert x \rVert_\nu$,
for some $C_l, C_u > 0$, $x \in \mathbb{R}^d$. Hence it is sufficient to work
with errors that are bounded in the weighted max-norm sense. In $(AV3)$ we assume $\limsup \limits_{n \to \infty}\ \lVert \epsilon_n \rVert_\nu \le \epsilon$ a.s.,
while in $(A1)$ we assume $\limsup \limits_{n \to \infty}\ \lVert \epsilon_n \rVert \le \epsilon$ a.s. Since $B^\epsilon := \{ y \mid \lVert y \rVert_\nu \le \epsilon\}$ is a convex compact subset of 
 $\mathbb{R}^d$ (see \textit{Lemma 7.2} of \cite{ramaswamy2017}), the analyses presented in Sections~\ref{sec_convergence} and \ref{sec_stability} carry forward verbatim, with $B^\epsilon$
 replacing $B_\epsilon(0)$.
% \begin{claim}
%  \label{a2vi_be}
%   $B^\epsilon := \{ y \mid \lVert y \rVert_\nu \le \epsilon\}$ is a convex compact subset of 
%  $\mathbb{R}^d$, where $\epsilon > 0$.
% \end{claim}
% \begin{proof}
%  For a proof of the above claim the reader is referred to the proof of \textit{Lemma 7.2} in \cite{ramaswamy2017}.
% \end{proof}

It follows directly from $(AV2)$ that $(S4a)$ is satisfied. If we show that (\ref{a2vi_a2vi}) also satisfies $(S5)$, then
we may conclude that the
iterates are stable and convergent. For this purpose, we compare the iterates $\{J_n\}_{n \ge 0}$, from (\ref{a2vi_a2vi}),
to their projective counterparts $\{\hat{J}_n \}_{n \ge 0}$. We can show that $\hat{J}_n \to \mathbb{A}$, where 
$\mathbb{A}$ is an attractor of $\dot{x}(t) \in \mu^* (TJ(t) - J(t)) + B^\epsilon$, contained within a small neighborhood of $J^*$ and $\mu^* = diag(1/d, \ldots, 1/d)$.
This neighborhood is dependent on the approximation errors.
Since $\hat{J}_n \to \mathbb{A}$, $\exists N$, possibly sample path dependent, such that $\hat{J}_n \in \mathcal{C}$ for all $n \ge N$. Following the arguments
presented in the proof of \textit{Theorem 3} in \cite{ramaswamy2017} we can show that
\[
 \lVert J_n - \hat{J}_n \rVert_\nu \le \lVert J_N - \hat{J}_N \rVert_\nu \vee \left( \frac{2 \epsilon}{1 - \alpha} \right),
\]
where $\alpha$ is the ``contraction constant'' associated with the Bellman operator $T$. In other words, we get that (\ref{a2vi_a2vi}) satisfies 
$(S5)$. Supposing balanced step-sizes are used, the following result is immediate.
\begin{theorem}
 \label{a2vi_main}
  Under $(AV1)$-$(AV3)$, $(A5)$, $(S1)$ and $(S2)$, (\ref{a2vi_a2vi}) is stable and converges to
  some point in $\left \{J \mid \lVert TJ - J \rVert_\nu \le d \epsilon \right\}$, where $\epsilon$
  is the norm-bound on the approximation errors.
\end{theorem}
\begin{proof}
 From the above discussion, it is clear that A2VI is bounded a.s. (stable). Since balanced step-sizes are used, to study the long-term
 behavior of A2VI one needs to study
 $\dot{J}(t) \in \mu^* ((TJ)(t) - J(t)) + B^\epsilon$, where $\mu^* = diag(1/d, \ldots, 1/d)$. 
 It follows from \textit{Theorem 2}
 of \textit{Chapter 6} in \cite{aubin2012differential} that any solution to the aforementioned DI will converge to an equilibrium point of
 $T(\cdotp) + B^{d \epsilon}$, where $B^{d \epsilon} :=  \{d x \mid x \in B^\epsilon \}$. This is because
 $\dot{J}(t) \in \mu^* ((TJ)(t) - J(t) + B^{d \epsilon})$ and $\dot{J}(t) \in TJ(t) - J(t) + B^{d \epsilon}$ are qualitatively similar and
 only differ in scale. The equilibrium points of $T + B^{d \epsilon}$ are given by $\left \{J \mid \lVert TJ - J \rVert_\nu \le d \epsilon \right\}$.
 For more details the reader is referred to \textit{Section 7} of \cite{ramaswamy2017}.
\end{proof}
We have shown that A2VI is stable as long as the approximation errors are asymptotically bounded. We do not distinguish between biased and unbiased errors.
Further, we show that A2VI converges to a fixed point of a scaling of the perturbed Bellman operator $\mu^* TJ + B^\epsilon$.
%%%%%%%%%%%%%%%%%%%%%%%%%%%%%%%%%%%%%%%%%%%%%%%%%%%%%%%%%%%%%%%%%%%%%%%%%%%%%%%%%%%%%%%%%%%%%%%%%%%%%%%%%%%%%%%%%%%%%%%%%%%%%%%%%%%%%%%%%%%%%%%%
%%%%%%%%%%%%%%%%%%%%%%%%%%%%%%%%%%%%%%%%%%%%%%%%%%%%%%%%%%%%%%%%%%%%%%%%%%%%%%%%%%%%%%%%%%%%%%%%%%%%%%%%%%%%%%%%%%%%%%%%%%%%%%%%%%%%%%%%%%%%%%%%
\subsection{Asynchronous approximate policy gradient iteration (A2PG)} \label{sec_a2pi}
Policy gradient method is an important reinforcement learning algorithm developed by Sutton \textit{et al.}, in 2000 \cite{sutton2000}. 
This method relies on a parameterization of the policy space, say $\pi(\theta)$. This parameterization is typically through the use of a deep
neural network. 
Once a parameterization is determined, one seeks out a local minimizer
$\hat{\theta}$ in the parameter space, in order to find the optimal policy. However, 
there are several situations wherein one either cannot calculate or does not wish to calculate
the exact gradient $\nabla_\theta \pi(\theta_n)$ at every stage. This could be due to the use of a non-differentiable activation function or it
could be a consequence of using gradient estimators such as 
$SPSA$-$C$ \cite{ramaswamy2017analysis} (simultaneous perturbation stochastic approximation
with a constant sensitivity parameter) or other finite difference methods.
In these cases, one has to deal with a policy gradient scheme with non-diminishing approximation errors. 
In the present work, we are interested in policy gradient methods within the setting
of large-scale distributed systems. A general form of approximate policy gradient methods which satisfy all these conditions is given below:
\begin{equation}
\label{a2pi_a2pi}
\begin{split}
 &\theta_{n+1}(i) = \theta_n (i)  - a(\nu(i,n)) I\{ i \in Y_n\} \\ 
 &\left( (\mathcal{A}\nabla_\theta \pi)_i(\theta_{n - \tau_{1i}(n)}(1), \ldots, 
\theta_{n - \tau_{di}(n)}(d))
   + M_{n+1}(i) \right).
   \end{split}
\end{equation}
We call the above scheme as asynchronous approximate policy gradient iteration or A2PG.
As in Section~\ref{sec_a2vi}, we can impose natural conditions on the gradient ($\nabla \pi(\cdotp)$), 
the noise and other parameters of (\ref{a2pi_a2pi}).
Suppose the approximation errors are asymptotically bounded, then one can show that the iterates converge to a neighborhood of some
local minimizer $\hat{\theta}$. Further, this neighborhood is a function of the approximation errors.
For details on the relationship between the neighborhood and approximation errors, the reader is referred to \cite{ramaswamy2017analysis}.
%%%%%%%%%%%%%%%%%%%%%%%%%%%%%%%%%%%%%%%%%%%%
%%%%%%%%%%%%%%%%%%%%%%%%%%%%%%%%%%%%%%%%%%%%
\subsection{Experimental results} \label{sec_exp}
In this section, we consider an asynchronous algorithm (given by eq. \eqref{asmp_aasaa}) to find the minimum of $F: \mathbb{R}^d \to \mathbb{R}^d$, where $d \ge 2$. The function $F$ is defined as $F(x_1, \ldots, x_d) :=$ $\left(F_1(x_1, \ldots, x_d), \ldots,  F_d(x_1, \ldots, x_d)\right)$, where $F_1, \ldots F_d: \mathbb{R}^d \to \mathbb{R}$.

\noindent
{\bf [Experimental set-up]} For better exposition, we consider an iteration in dimension $2$, i.e., $d = 2$. The function $F$ is defined as follows: $F_1(x) := \frac{1}{2} (x^{ T}Ax) (1)$, $F_2(x) := \frac{1}{2} (x^{ T}Bx) (2)$ and $F(x) := (F_1(x), F_2(x))$. The matrices $A$ and $B$ are randomly constructed positive definite matrices of dimension $2 \times 2$. A random error vector of norm less than $\epsilon \neq 0$ is added to the gradient at every step. Each component of this error vector is independent and uniformly distributed in $[0, \epsilon/ 2]$. It may be noted that $\nabla _x F_1(x) = Ax$ and $\nabla _x F_2(x) = Bx$ for $x \in \mathbb{R}^d$.

\noindent
Agent-1 runs the following:
\[
x_{n+1}(1) = x_n (1) - a(n) \left[ A \begin{bmatrix} x_n(1) \\ x_{n - \tau_{2,1}}(2) \end{bmatrix}(1)  + \epsilon_1 \right],
\]
while agent-2 runs the following:
\[
x_{n+1}(2) = x_n (2) - a(n) \left[ B \begin{bmatrix} x_{n - \tau_{2,1}}(1) \\ x_n(2) \end{bmatrix}(2)  + \epsilon_2 \right].
\]
The above distributed algorithm was run for $1000$ iterations using the step-size seqeunce $\{\nicefrac{1}{(n+10)}\} _{n=1} ^{1000}$. Since the matrices $A$ and $B$ are positive definite, we expect the limit to be the origin. 

At every step the two agents exchange (state) information with probability $p_c$, and with probability $1 - p_c$ the agents use old (state) information. In other words, $p_c$ represents the communication probability in our experiments. Note that we have used symmetric delays for simplicity and the experiments can be easily repeated with asymmetric delays.

Results from the experiments are summarized in the two figures below. In both the figures, $\lVert (\epsilon_1, \epsilon_2) \rVert$ is plotted along $x$-axis and $\log( \lVert x_{1000} \rVert)$ is plotted along $y$-axis. In other words, any point in the plot represents $(\lVert (\epsilon_1, \epsilon_2) \rVert,\ \log( \lVert x_{1000} \rVert))$. Each figure has five differently colored line graphs to represent the five sample runs of the algorithm. For each sample run, the parameters $(x_1(1), x_1(2))$ (initial point) and the matrices $A$ and $B$ are randomly chosen, and the norm bound on additive errors $(\epsilon_1$, $\epsilon_2)$ is varied from $0.2$ to $3$ in steps of 0.1. Fig. \ref{fig05} illustrates all the experiments with $p_c = 0.4$, and Fig. \ref{fig08} illustrates all the experiments with $p_c = 0.8$.
\begin{figure}[H]
\centering
\includegraphics[width=2.5in]{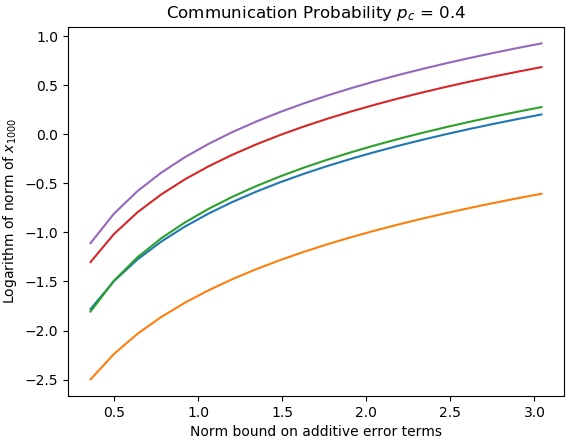}
\caption{Five random sample runs with $p_c = 0.4.$ $\lVert (\epsilon_1, \epsilon_2) \rVert$ is plotted along $x$-axis and $\log( \lVert x_{1000} \rVert)$ is plotted along $y$-axis.}
\label{fig05}
\end{figure}
\begin{figure}[H]
\centering
\includegraphics[width=2.5in]{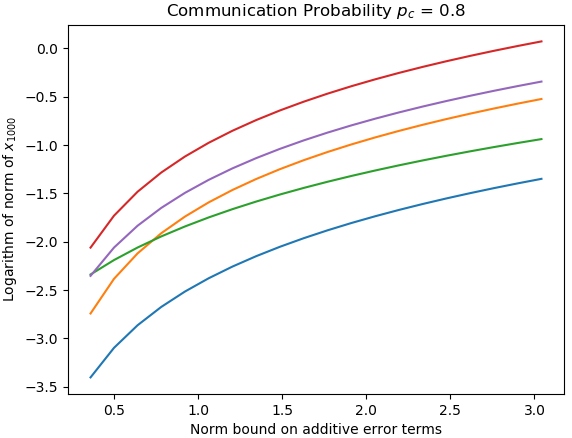}
\caption{Five random sample runs with $p_c = 0.8.$ $\lVert (\epsilon_1, \epsilon_2) \rVert$ is plotted along $x$-axis and $\log( \lVert x_{1000} \rVert)$ is plotted along $y$-axis.}
\label{fig08}
\end{figure}
In Figures \ref{fig05} and \ref{fig08}, the agents exchange data $40\%$ and $80\%$ of the time, respectively. It can be seen that the algorithm converged farther from the origin when the additive errors are larger. When $p_c = 0.8$, the algorithm converged to a point close to the origin, even for large additive errors, as compared to $p_c = 0.4$. {\bf The experiments seem to suggest that frequent communications should be used to counter the effect of biased additive errors.}
%%%%%%%%%%%%%%%%%%%%%%%%%%%%%%%%%%%%%%%%%%%%%%%%%%%%%%%%%%%%%%%%%%%%%%%%%%%%%%%%%%%%%%%%%%%%%%%%%%%%%%%%%%%%%%%%%%%%%%%%%%%%%%%%%%%%%%%%%%%%%%%%
\section{Verifiability of assumption $(S5)$} \label{sec_verify}
In this section, we address the verifiability of assumption $(S5)$. We do not discuss other assumptions, since they deal with the objective function, step-sizes or noise, in a manner that is standard to literature. However, to ensure $(S5)$, one needs to compare the algorithm iterates with a projective scheme. Further, the experimenter is typically uninterested in the projective scheme itself. In this section, we show that $(S5)$ is satisfied for fixed point finding algorithms such as A2VI, provided the objective function is non-expansive.

Recall A2VI and its projective counterpart:
\begin{equation} \label{verify_iterate}
\begin{split}
J_{n+1} &= J_n + a(n) D_n \left[ TJ_n - J_n + \epsilon_n \right] ,\\
\hat{J}_{n+1} &\in \pro_{\mathcal{B}, \mathcal{C}} \left(\hat{J}_n + a(n) D_n \left[ T\hat{J}_n - J_n + \hat{\epsilon}_n \right] \right).
\end{split}
\end{equation}
Unlike in Section $\ref{sec_a2vi}$, we assume here that $T$ is non-expansive with respect to some norm $p$, i.e., $p(Tx - Ty) \le p(x - y)$. For all $x, y$ it follows from Lemma \ref{projective_gis0_2} that the projective scheme converges to $\mathbb{A}$ almost surely. In other words, there exists a sample path dependent $N$ such that $\{\hat{J}_n\}_{n \ge N} \subseteq \mathcal{C}$ a.s. Further, $\hat{J}_{n+1} = \hat{J}_n + a(n) D_n \left[ T\hat{J}_n + \hat{\epsilon}_n \right]$ for all $n \ge N$. For $n \ge N$, first, we take the difference between the two iterations in \eqref{verify_iterate}. Then, we take the norms on both sides, to get the following:
\begin{equation*}
\begin{split}
p(J_{n +1} - \hat{J}_{n +1}) &\le (1 - a(n)) p(J_n  - \hat{J}_n) \\ &+ a(n) p (TJ_n - T\hat{J}_n) + a(n) p(\epsilon_n - \hat{\epsilon}_n)
\end{split}
\end{equation*}
Since $T$ is non-expansive we get:
\[
p(J_{n +1} - \hat{J}_{n +1}) \le p(J_{n} - \hat{J}_{n}) + a(n) p(\epsilon_n - \hat{\epsilon}_n).
\]
For $k \ge 1$, we have:
\[
p(J_{N +k} - \hat{J}_{N +k}) \le p(J_{N} - \hat{J}_{N}) + \sum \limits_{n=N}^{N + k-1}a(n) p(\epsilon_{n} - \hat{\epsilon}_n).
\]
As long as $p(\epsilon_{n} - \hat{\epsilon}_n) \in o(a(n))$ ($o(\cdotp)$, we get:
\[
p(J_{N +k} - \hat{J}_{N +k}) \le p(J_{N} - \hat{J}_{N}) + \sum \limits_{n=0}^{\infty}a(n)^2 < \infty.
\]
It may be noted that many important RL and MDP algorithms such as Q-learning and Value Iteration are fixed point finding algorithms. In \cite{abounadi}, the objective function of Q-learning is shown to be non-expansive. To summarize, the above set of arguments can be used to verify S(5) for approximate asynchronous fixed point finding algorithms with non-expansive objective functions.
%%%%%%%%%%%%%%%%%%%%%%%%%%%%%%%%%%%%%%%%%%%%%%%%%%%%%%%%%%%%%%%%%%%%%%%%%%%%%%%%%%%%%%%%%%%%%%%%%%%%%%%%%%%%%%%%%%%%%%%%%%%%%%%%%%%%%%%%%%%%%%%%
%%%%%%%%%%%%%%%%%%%%%%%%%%%%%%%%%%%%%%%%%%%%%%%%%%%%%%%%%%%%%%%%%%%%%%%%%%%%%%%%%%%%%%%%%%%%%%%%%%%%%%%%%%%%%%%%%%%%%%%%%%%%%%%%%%%%%%%%%%%%%%%%
%%%%%%%%%%%%%%%%%%%%%%%%%%%%%%%%%%%%%%%%%%%%%%%%%%%%%%%%%%%%%%%%%%%%%%%%%%%%%%%%%%%%%%%%%%%%%%%%%%%%%%%%%%%%%%%%%%%%%%%%%%%%%%%%%%%%%%%%%%%%%%%%
\section{Summary of our contributions and conclusions} \label{sec_conclusion}
In this paper, we considered a natural extension of asynchronous stochastic approximation algorithms that accommodates the use of function approximations.
For this purpose, we considered asynchronous stochastic approximations 
with asymptotically bounded, and possibly biased, approximation errors.
The assumptions and the analyses presented are motivated by the need to understand the current crop of deep reinforcement learning algorithms.
We are particularly interested in these algorithms when used within the setting of multi-agent learning and control.

Our framework allows for complete asynchronicity in that each agent is
guided by its own local clock. 
Although the agents are fully asynchronous, we require that the agents are updated,
roughly, the same number of times, in the long run.
 Our framework can be used to
analyze asynchronous approximate value iteration (A2VI).
A2VI is an adaptation of regular value iteration with noise to the setting of large-scale multi-agent learning and control.
Here, we showed that 
A2VI converges to a fixed point of the perturbed Bellman operator when balanced step-sizes are used. We also established a relationship between these fixed points
and the approximation errors. Note that the use of function approximators required us to consider the perturbed Bellman operator.
We further analyzed a similar adaptation, A2PG, of the classical policy gradient iteration to the multi-agent setting.
 We briefly discussed how A2PG converges to a small neighborhood of local minima of the parameterized policy function.
Again, this neighborhood is directly related to the approximation errors.

 \textit{An important consequence of our theory is the following: stability of the aforementioned algorithms remains unaffected when the 
approximation errors are asymptotically bounded,
although possibly biased. Since a function approximator (eg. DNN) is continuously trained, it is reasonable to expect the errors to diminish asymptotically, even
though they may not vanish completely.}
\textit{It is worth noting that ours is one of the first theoretical results that can be used to understand
the long-term behavior of deep reinforcement learning algorithms within the setting of multi-agent learning and control.}

 In the future, we would like to make a two-fold extension to our analysis: (i) Allow for multiple timescales and (ii) allow for objective functions that 
are driven by controlled Markov processes. This will help us analyze other popular algorithms such as Deep Q-Network, deep temporal difference learning and deep deterministic policy gradient (a popular actor-critic algorithm).
When implementing DeepRL algorithms in an online setting, the learning rate is generally fixed. To this end, we would also wish to explore one and two timescale algorithms with
constant step-sizes and function approximations.% Specifically, we showed that the algorithm is stable
% (bounded almost surely) provided the approximation errors due to deep function approximators are asymptotically bounded. 
%%%%%%%%%%%%%%%%%%%%%%%%%%%%%%%%%%%%%%%%%%%%%%%%%%%%%%%%%%%%%%%%%%%%%%%%%%%%%%%%%%%%%%%%%%%%%%%%%%%%%%%%%%%%%%%%%%%%%%%%%%%%%%%%%%%%%%%%%%%%%%%%
%%%%%%%%%%%%%%%%%%%%%%%%%%%%%%%%%%%%%%%%%%%%%%%%%%%%%%%%%%%%%%%%%%%%%%%%%%%%%%%%%%%%%%%%%%%%%%%%%%%%%%%%%%%%%%%%%%%%%%%%%%%%%%%%%%%%%%%%%%%%%%%%
%%%%%%%%%%%%%%%%%%%%%%%%%%%%%%%%%%%%%%%%%%%%%%%%%%%%%%%%%%%%%%%%%%%%%%%%%%%%%%%%%%%%%%%%%%%%%%%%%%%%%%%%%%%%%%%%%%%%%%%%%%%%%%%%%%%%%%%%%%%%%%%%
%%%%%%%%%%%%%%%%%%%%%%%%%%%%%%%%%%%%%%%%%%%%%%%%%%%%%%%%%%%%%%%%%%%%%%%%%%%%%%%%%%%%%%%%%%%%%%%%%%%%%%%%%%%%%%%%%%%%%%%%%%%%%%%%%%%%%%%%%%%%%%%%
%%%%%%%%%%%%%%%%%%%%%%%%%%%%%%%%%%%%%%%%%%%%%%%%%%%%%%%%%%%%%%%%%%%%%%%%%%%%%%%%%%%%%%%%%%%%%%%%%%%%%%%%%%%%%%%%%%%%%%%%%%%%%%%%%%%%%%%%%%%%%%%%
%%%%%%%%%%%%%%%%%%%%%%%%%%%%%%%%%%%%%%%%%%%%%%%%%%%%%%%%%%%%%%%%%%%%%%%%%%%%%%%%%%%%%%%%%%%%%%%%%%%%%%%%%%%%%%%%%%%%%%%%%%%%%%%%%%%%%%%%%%%%%%%%
%%%%%%%%%%%%%%%%%%%%%%%%%%%%%%%%%%%%%%%%%%%%%%%%%%%%%%%%%%%%%%%%%%%%%%%%%%%%%%%%%%%%%%%%%%%%%%%%%%%%%%%%%%%%%%%%%%%%%%%%%%%%%%%%%%%%%%%%%%%%%%%%
\bibliographystyle{plain}
\bibliography{AAVI}

\begin{thebibliography}{10}

\bibitem{abounadi}
J.~Abounadi, D.P. Bertsekas, and V.~Borkar.
\newblock Stochastic approximation for nonexpansive maps: Application to
  q-learning algorithms.
\newblock {\em SIAM Journal on Control and Optimization}, 41(1):1--22, 2002.

\bibitem{aubin2012differential}
J-P. Aubin and A.~Cellina.
\newblock {\em Differential inclusions: set-valued maps and viability theory},
  volume 264.
\newblock Springer Science \& Business Media, 2012.

\bibitem{Benaim96}
M.~Bena{\" i}m.
\newblock A dynamical system approach to stochastic approximations.
\newblock {\em SIAM J. Control Optim.}, 34(2):437--472, 1996.

\bibitem{BenaimHirsch}
M.~Bena{\"i}m and M.~W. Hirsch.
\newblock Asymptotic pseudotrajectories and chain recurrent flows, with
  applications.
\newblock {\em J. Dynam. Differential Equations}, 8:141--176, 1996.

\bibitem{Benaim05}
M.~Bena\"{i}m, J.~Hofbauer, and S.~Sorin.
\newblock Stochastic approximations and differential inclusions.
\newblock {\em SIAM Journal on Control and Optimization}, pages 328--348, 2005.

\bibitem{benaim2006stochastic}
M.~Bena{\"\i}m, J.~Hofbauer, and S.~Sorin.
\newblock Stochastic approximations and differential inclusions, part ii:
  Applications.
\newblock {\em Mathematics of Operations Research}, 31(4):673--695, 2006.

\bibitem{BertsekasBook}
D.P. Bertsekas and J.N. Tsitsiklis.
\newblock {\em Neuro-Dynamic Programming}.
\newblock Athena Scientific, 1st edition, 1996.

\bibitem{Bhatnagar}
S.~Bhatnagar.
\newblock The {Borkar-Meyn} theorem for asynchronous stochastic approximations.
\newblock {\em Systems {\&} Control Letters}, 60(7):472--478, 2011.

\bibitem{bianchi2019constant}
P.~Bianchi, W.~Hachem, and A.~Salim.
\newblock Constant step stochastic approximations involving differential
  inclusions: Stability, long-run convergence and applications.
\newblock {\em Stochastics}, 91(2):288--320, 2019.

\bibitem{Billingsley}
P.~Billingsley.
\newblock {\em Convergence of probability measures}.
\newblock John Wiley \& Sons, 2013.

\bibitem{BorkarBook}
V.~S. Borkar.
\newblock {\em Stochastic Approximation: A Dynamical Systems Viewpoint}.
\newblock Cambridge University Press, 2008.

\bibitem{Borkar99}
V.~S. Borkar and S.P. Meyn.
\newblock The {O.D.E.} method for convergence of stochastic approximation and
  reinforcement learning.
\newblock {\em SIAM J. Control Optim}, 38:447--469, 1999.

\bibitem{Borkar_asynchronous}
V.S. Borkar.
\newblock Asynchronous stochastic approximations.
\newblock {\em SIAM Journal on Control and Optimization}, 36(3):840--851, 1998.

\bibitem{KushnerYin}
H.~Kushner and G.G. Yin.
\newblock {\em Stochastic Approximation and Recursive Algorithms and
  Applications}.
\newblock Springer, 2003.

\bibitem{li17}
Y.~Li.
\newblock Deep reinforcement learning: An overview.
\newblock {\em arXiv preprint arXiv:1701.07274}, 2017.

\bibitem{mnih}
V.~Mnih, K.~Kavukcuoglu, D.~Silver, A.A. Rusu, J.~Veness, M.G. Bellemare,
  A.~Graves, M.~Riedmiller, A.K. Fidjeland, G.~Ostrovski, et~al.
\newblock Human-level control through deep reinforcement learning.
\newblock {\em Nature}, 518(7540):529, 2015.

\bibitem{mniha}
Volodymyr Mnih, Adria~Puigdomenech Badia, Mehdi Mirza, Alex Graves, Timothy
  Lillicrap, Tim Harley, David Silver, and Koray Kavukcuoglu.
\newblock Asynchronous methods for deep reinforcement learning.
\newblock In {\em International Conference on Machine Learning}, pages
  1928--1937, 2016.

\bibitem{munos03}
R.~Munos.
\newblock Error bounds for approximate policy iteration.
\newblock In {\em ICML}, volume~3, pages 560--567, 2003.

\bibitem{munos}
R.~Munos.
\newblock Error bounds for approximate value iteration.
\newblock In {\em Proceedings of the National Conference on Artificial
  Intelligence}, volume~20, page 1006, 2005.

\bibitem{perkins}
S.~Perkins and D.S. Leslie.
\newblock Asynchronous stochastic approximation with differential inclusions.
\newblock {\em Stochastic Systems}, 2(2):409--446, 2013.

\bibitem{ramaswamy2017analysis}
A.~Ramaswamy and S.~Bhatnagar.
\newblock Analysis of gradient descent methods with non-diminishing, bounded
  errors.
\newblock {\em IEEE Transactions on Automatic Control}, 2017.
\newblock doi:10.1109/TAC.2017.2744598.

\bibitem{ramaswamy2017}
A.~Ramaswamy and S.~Bhatnagar.
\newblock Conditions for stability and convergence of set-valued stochastic
  approximations: Applications to approximate value and fixed point iterations.
\newblock {\em arXiv preprint arXiv:1709.04673}, 2017.

\bibitem{Ramaswamy}
A.~Ramaswamy and S.~Bhatnagar.
\newblock A generalization of the {Borkar-Meyn} theorem for stochastic
  recursive inclusions.
\newblock {\em Mathematics of Operations Research}, 42(3):648--661, 2017.

\bibitem{robbins}
H.~Robbins and S.~Monro.
\newblock A stochastic approximation method.
\newblock {\em The annals of mathematical statistics}, pages 400--407, 1951.

\bibitem{sutton2000}
R.S. Sutton, D.~A. McAllester, S.P. Singh, and Y.~Mansour.
\newblock Policy gradient methods for reinforcement learning with function
  approximation.
\newblock In {\em Advances in neural information processing systems}, pages
  1057--1063, 2000.

\bibitem{tamar}
A.~Tamar, Y.~Wu, G.~Thomas, S.~Levine, and P.~Abbeel.
\newblock Value iteration networks.
\newblock In {\em Advances in Neural Information Processing Systems}, pages
  2154--2162, 2016.

\end{thebibliography}
\end{document}